\documentclass{amsart}
\usepackage{amsmath,amssymb,amsthm,mathrsfs}

\usepackage{pgf,tikz}
\usepackage{mathrsfs}
\usepackage{caption,subcaption}
\usepackage{graphicx}
\usepackage{arydshln}
\usepackage[colorlinks,linkcolor=blue,citecolor=blue]{hyperref}
\usetikzlibrary{arrows}
\definecolor{qqqqff}{rgb}{0.,0.,1.}
\definecolor{ffqqqq}{rgb}{1.,0.,0.}
\definecolor{qqffff}{rgb}{0.,1.,1.}
\definecolor{ffffff}{rgb}{1.,1.,1.}

\theoremstyle{plain}
\newtheorem{theorem}{Theorem}[section]
\newtheorem{lemma}[theorem]{Lemma}
\newtheorem{proposition}[theorem]{Proposition}

\newtheorem{corollary}[theorem]{Corollary}

\theoremstyle{definition}

\numberwithin{equation}{section}
\allowdisplaybreaks[1]

\newcommand{\B}{{\mathbb B}}
\newcommand{\BC}{{\mathbb C}}
\newcommand{\BF}{{\mathbb F}}

\newcommand{\BR}{{\mathbb R}}

\newcommand{\cF}{{\mathcal F}}
\newcommand{\cH}{{\mathcal H}}

\newcommand{\cK}{{\mathcal K}}\newcommand{\cL}{{\mathcal L}}

\newcommand{\cU}{{\mathcal U}}
\newcommand{\cX}{{\mathcal X}}
\newcommand{\cY}{{\mathcal Y}}\newcommand{\cZ}{{\mathcal Z}}

\newcommand{\la}{\lambda}

\newcommand{\rank}{\operatorname{rank}}

\newcommand{\ord}{\operatorname{ord}}

\newcommand{\im}{\operatorname{Im}}

\newcommand{\kr}{\operatorname{Ker}}
\newcommand{\diag}{\operatorname{diag}}

\newcommand{\PSD}{\operatorname{PSD}}

\newcommand{\spn}{\operatorname{span}}

\newcommand{\coker}{\operatorname{Coker}}

\newcommand{\mat}[2]{\ensuremath{\left[\begin{array}{#1}#2\end{array}\right]}}

\newcommand{\sbm}[1]{\left[\begin{smallmatrix}#1\end{smallmatrix}\right]}
\newcommand{\ol}[1]{{\overline{#1}}}

\newcommand{\wtil}[1]{\widetilde{#1}}

\newcommand{\what}[1]{\widehat{#1}}

\newcommand{\BB}[1]{\mathbb{#1}}

\newcommand{\cal}[1]{\mathcal{#1}}

\newcommand{\half}{\frac{1}{2}}
\newcommand{\ands}{\quad\mbox{and}\quad}

\begin{document}

\title{Graphs with sparsity order at most two: The complex case}

\author{S. ter Horst}
\address{S. ter Horst, Department of Mathematics, Unit for BMI,
North-West University, Potchefstroom 2531, South Africa}
\email{Sanne.TerHorst@nwu.ac.za}

\author{E.M. Klem}
\address{E.M. Klem, Department of Mathematics, Unit for BMI, North-West University,
Potchefstroom, 2531 South Africa}
\email{estiaanklem@gmail.com}

\thanks{This work is based on the research supported in part by the National Research Foundation of South Africa (Grant Numbers 93406, and 89833).}

\begin{abstract}
The sparsity order of a (simple undirected) graph is the highest possible rank (over $\BR$ or $\BC$) of the extremal elements in the matrix cone that consists of positive semidefinite matrices with prescribed zeros on the positions that correspond to non-edges of the graph (excluding the diagonal entries). The graphs of sparsity order 1 (for both $\BR$ and $\BC$) correspond to chordal graphs, those graphs that do not contain a cycle of length greater than three, as an induced subgraph, or equivalently, is a clique-sum of cliques. There exist analogues, though more complicated, characterizations of the case where the sparsity order is at most 2, which are different for $\BR$ and $\BC$. The existing proof for the complex case, is based on the result for the real case. In this paper we provide a more elementary proof of the characterization of the graphs whose complex sparsity order is at most two. Part of our proof relies on a characterization of the $\{P_4,\ol{K}_3\}$-free graphs, with $P_4$ the path of length 3 and $\ol{K}_3$ the stable set of cardinality 3, and of the class of clique-sums of such graphs.
\end{abstract}



\maketitle


\section{Introduction}\label{S:intro}

Let $G=(V,E)$ be a simple undirected graph with vertex set $V=\{1,\ldots n\}$ and edge set $E\subset V\times V$. Let $\BB{F}$ be either $\BB{C}$ or $\BB{R}$. We write $\cal{H}_G$ for the linear space over $\BB{R}$ that consists of Hermitian $n\times n$ matrices over $\BB{F}$ (hence symmetric in case $\BB{F}=\BB{R}$) with the property that for $i\neq j$ the $(i,j)$-th entry is equal to 0 whenever $(i,j)\notin E$. With $\PSD_G$ we denote the set of positive semidefinite matrices in $\cal{H}_G$. Then $\PSD_G$ forms a convex cone in $\cal{H}_G$, i.e., $\PSD_G$ is closed under sums and multiplication by positive scalars. Furthermore, any element of $\PSD_G$ can be written as a finite sum of extremal elements of $\PSD_G$. Recall that an element $X$ of $\PSD_G$ is called {\em extremal} whenever the only way to have $X=X'+X''$ with $X',X''\in \PSD_G$ is when $X'$ and $X''$ are positive scalar multiples of $X$. In particular, the only way to write an extremal element $X$ of $\PSD_G$ as a sum of extremals, is if all extremals are positive scalar multiples of $X$ itself. The {\em sparsity order} of the graph $G$, as introduced in \cite{AHMR88}, is defined to be the maximum rank among all extremal elements of $\PSD_G$ and is denoted by $\ord_\BB{F}(G)$. In other words, any element of $\PSD_G$ can be written as a sum of (extremal) elements of $\PSD_G$ with rank at most $\ord_\BB{F}(G)$, and there are elements of $\PSD_G$ for which this is the best that can be achieved.

The notation and terminology used in this paper is mostly standard. However, at the end of the introduction we will recall some graph theory preliminaries, mostly notation, used throughout the paper.\medskip

\paragraph{\bf Sparsity order 1}
In case $G$ is a complete graph $K_n$, i.e., $E=V\times V$, $\PSD_G$ is simply the cone of $n\times n$ positive semidefinite matrices, and using the singular value decomposition it is easily seen that any positive semidefinite matrix can be written as a sum of rank-one positive semidefinite matrices. Hence the sparsity order is equal to 1, for both $\BB{F}=\BB{C}$ and $\BB{F}=\BB{R}$.

The complete characterization of graphs that have sparsity order 1 was presented in \cite{AHMR88}. It turns out that, again for both $\BB{F}=\BB{C}$ and $\BB{F}=\BB{R}$, a graph $G$ has sparsity order equal to 1 if and only if $G$ is chordal, that is, if and only if $G$ has no induced cycles of length greater than three. This result is closely connected to the positive semidefinite matrix completion problem. The connection between the latter problem and chordal graphs was first observed in \cite{GJSW84}. Viewing $\cal{H}_G$ as a Hilbert space over $\BB{R}$ with the trace inner product, the dual cone of $\PSD_G$ in $\cal{H}_G$ consists precisely of the elements of $\cal{H}_G$ that can be `completed' to a positive semidefinite matrix by replacing the prescribed zeros with other elements of $\BB{F}$. For this to work, it is essential that the elements of $\PSD_G$ can be written as sums of rank one elements from $\PSD_G$. See \cite{BW11} for a recent account, and the references given there for further background. In view of the results to come, we mention here that $G$ is chordal if and only if $G$ can be written as a clique-sum of cliques; a result that goes back to work of Dirac \cite{D61}. The results discussed above that are relevant in the sequel are listed in the following theorem.

\begin{theorem}\label{T:SO1}
Let $G$ be a simple undirected graph. For both $\BB{F}=\BB{C}$ and $\BB{F}=\BB{R}$ the following are equivalent:
\begin{itemize}
	\item[(i)] $G$ has sparsity order 1, i.e., $\ord_\BB{F}(G)=1$;

	\item[(ii)] $G$ is chordal, i.e., for $n\geq 4$ the cycle $C_n$ does not appear as an induced subgraph of $G$;

	\item[(iii)] $G$ is a clique-sum of cliques.
\end{itemize}
\end{theorem}

\paragraph{\bf Sparsity order at most 2}
There exists a characterization for the case where $\ord_\BB{F}(G)\leq 2,$ analogous to that for $\ord_\BB{F}(G)=1$ in Theorem \ref{T:SO1}. However, this result is significantly more complicated, and the characterizations differ for the cases $\BB{F}=\BB{C}$ and $\BB{F}=\BB{R}$. The following result was obtained by Laurent in \cite[Theorem 13]{L01} for the case $\BB{F}=\BB{C}$.

\begin{theorem}\label{T:SO2complex}
Let $G$ be a simple undirected graph. Then the following are equivalent:
	\begin{itemize}
	\item[(i)] $G$ has complex sparsity order less than or equal to 2, i.e., $\ord_\BB{C}(G)\leq 2$;

	\item[(ii)] $G$ does not contain the cycles $C_n,$ $n\geq 5,$ as induced subgraphs, nor any of the graphs $D_1-D_6$ depicted in Figure \ref{F:D1toD6}.

	\begin{figure}[ht]
		\centering
		\begin{tikzpicture}[scale=0.75,line cap=round,line join=round,>=triangle 45,x=1.0cm,y=1.0cm]
		\clip(0.5,0.) rectangle (15.,4.);
		\draw (1.,1.)-- (1.,1.5);
		\draw (1.,1.5)-- (1.,2.);
		\draw (1.,2.)-- (1.,2.5);
		\draw (1.,2.5)-- (1.,3.);
		\draw (1.,3.)-- (1.,3.5);
		\draw (2.5,2.)-- (3.,1.5);
		\draw (3.,1.5)-- (3.5,1.5);
		\draw (3.5,1.5)-- (4.,2.);
		\draw (4.,2.)-- (3.5,2.5);
		\draw (3.5,2.5)-- (3.,2.5);
		\draw (2.5,2.)-- (3.,2.5);
		\draw (5.5,2.)-- (5.5,1.5);
		\draw (5.5,1.5)-- (5.5,1.);
		\draw (5.5,2.)-- (5.5,2.5);
		\draw (6.,2.)-- (6.,1.5);
		\draw (7.5,1.5)-- (7.75,2.);
		\draw (7.75,2.)-- (8.,1.5);
		\draw (8.,1.5)-- (7.5,1.5);
		\draw (8.5,2.)-- (8.5,1.5);
		\draw (1.15,0.6741665014673589) node[anchor=north] {$\ol{D}_1$};
		\draw (3.4,0.6741665014673589) node[anchor=north] {$\ol{D}_2$};
		\draw (5.85,0.6741665014673589) node[anchor=north] {$\ol{D}_3$};
		\draw (8.15,0.6741665014673589) node[anchor=north] {$\ol{D}_4$};
		\draw (10.25,2.)-- (10.,1.5);
		\draw (10.25,2.)-- (10.5,1.5);
		\draw (11.,2.)-- (11.,1.5);
		\draw (11.5,2.)-- (11.5,1.5);
		\draw (13.,2.)-- (13.,1.5);
		\draw (13.5,2.)-- (13.5,1.5);
		\draw (14.,2.)-- (14.,1.5);
		\draw (14.5,2.)-- (14.5,1.5);
		\draw (10.9,0.6741665014673589) node[anchor=north] {$\ol{D}_5$};
		\draw (13.85,0.6741665014673589) node[anchor=north] {$\ol{D}_6$};
		\begin{scriptsize}
		\draw [fill=black] (1.,1.) circle (1.5pt);
		\draw [fill=black] (1.,1.5) circle (1.5pt);
		\draw [fill=black] (1.,2.) circle (1.5pt);
		\draw [fill=black] (1.,2.5) circle (1.5pt);
		\draw [fill=black] (1.,3.) circle (1.5pt);
		\draw [fill=black] (1.,3.5) circle (1.5pt);
		\draw [fill=black] (2.5,2.) circle (1.5pt);
		\draw [fill=black] (3.,1.5) circle (1.5pt);
		\draw [fill=black] (3.5,1.5) circle (1.5pt);
		\draw [fill=black] (4.,2.) circle (1.5pt);
		\draw [fill=black] (3.5,2.5) circle (1.5pt);
		\draw [fill=black] (3.,2.5) circle (1.5pt);
		\draw [fill=black] (5.5,2.) circle (1.5pt);
		\draw [fill=black] (5.5,1.5) circle (1.5pt);
		\draw [fill=black] (5.5,1.) circle (1.5pt);
		\draw [fill=black] (5.5,2.5) circle (1.5pt);
		\draw [fill=black] (6.,2.) circle (1.5pt);
		\draw [fill=black] (6.,1.5) circle (1.5pt);
		\draw [fill=black] (7.5,1.5) circle (1.5pt);
		\draw [fill=black] (7.75,2.) circle (1.5pt);
		\draw [fill=black] (8.,1.5) circle (1.5pt);
		\draw [fill=black] (8.5,2.) circle (1.5pt);
		\draw [fill=black] (8.5,1.5) circle (1.5pt);
		\draw [fill=black] (10.25,2.) circle (1.5pt);
		\draw [fill=black] (10.,1.5) circle (1.5pt);
		\draw [fill=black] (10.5,1.5) circle (1.5pt);
		\draw [fill=black] (11.,2.) circle (1.5pt);
		\draw [fill=black] (11.,1.5) circle (1.5pt);
		\draw [fill=black] (11.5,2.) circle (1.5pt);
		\draw [fill=black] (11.5,1.5) circle (1.5pt);
		\draw [fill=black] (13.,2.) circle (1.5pt);
		\draw [fill=black] (13.,1.5) circle (1.5pt);
		\draw [fill=black] (13.5,2.) circle (1.5pt);
		\draw [fill=black] (13.5,1.5) circle (1.5pt);
		\draw [fill=black] (14.,2.) circle (1.5pt);
		\draw [fill=black] (14.,1.5) circle (1.5pt);
		\draw [fill=black] (14.5,2.) circle (1.5pt);
		\draw [fill=black] (14.5,1.5) circle (1.5pt);
		\end{scriptsize}
		\end{tikzpicture}	
		\caption{The complementary graphs $\ol{D}_1-\ol{D}_6$}\label{F:D1toD6}
	\end{figure}
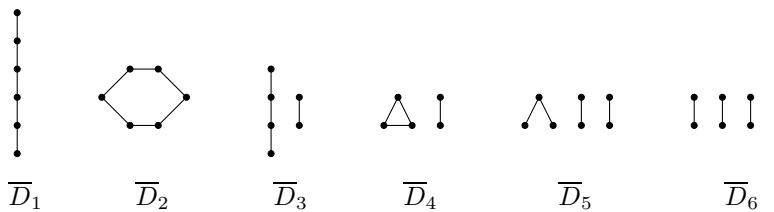
	
	\item[(iii)] $G$ is a clique-sum of graphs from the classes $\cal{K}_2$ and $\cal{K}_3$ depicted
in Figure \ref{F:K2K3}.

	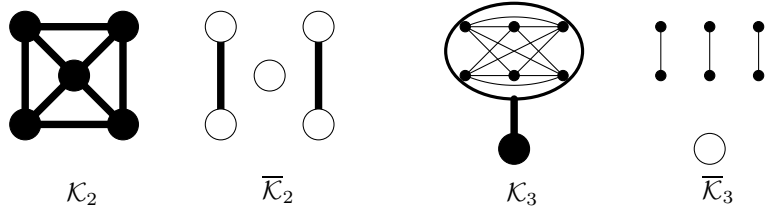
\begin{figure}[ht]
		\centering
		\begin{tikzpicture}[scale=1.3,line cap=round,line join=round,>=triangle 45,x=1.0cm,y=1.0cm]
		\clip(0.5,0.1) rectangle (9.,2.5);
		\draw [line width=2.8pt] (1.,2.)-- (1.5,1.5);
		\draw [line width=2.8pt] (1.5,1.5)-- (1.,1.);
		\draw [line width=2.8pt] (1.,2.)-- (1.,1.);
		\draw [line width=2.8pt] (1.,1.)-- (2.,1.);
		\draw [line width=2.8pt] (2.,1.)-- (1.5,1.5);
		\draw [line width=2.8pt] (2.,1.)-- (2.,2.);
		\draw [line width=2.8pt] (1.5,1.5)-- (2.,2.);
		\draw [line width=2.8pt] (2.,2.)-- (1.,2.);
		\draw [shift={(6.,0.8)}] plot[domain=1.1760052070951352:1.965587446494658,variable=\t]({1.*1.3*cos(\t r)+0.*1.3*sin(\t r)},{0.*1.3*cos(\t r)+1.*1.3*sin(\t r)});
		\draw [shift={(6.,2.7)}] plot[domain=4.317597860684927:5.1071801000844514,variable=\t]({1.*1.3*cos(\t r)+0.*1.3*sin(\t r)},{0.*1.3*cos(\t r)+1.*1.3*sin(\t r)});
		\draw [rotate around={0.:(6.,1.75)},line width=1.2pt] (6.,1.75) ellipse (0.7cm and 0.48989794855664553cm);
		\draw [line width=2.8pt] (6.,0.75)-- (6.,1.2601020514433545);		
		\draw (7.5,2.)-- (7.5,1.5);
		\draw (8.,2.)-- (8.,1.5);
		\draw (8.5,2.)-- (8.5,1.5);
		\draw (5.5,2.)-- (6.,1.5);
		\draw (5.5,1.5)-- (6.,2.);
		\draw (6.,2.)-- (6.5,1.5);
		\draw (6.5,1.5)-- (5.5,2.);
		\draw (5.5,2.)-- (6.,2.);
		\draw (6.,2.)-- (6.5,2.);
		\draw (6.5,2.)-- (6.,1.5);
		\draw (6.,1.5)-- (5.5,1.5);
		\draw (5.5,1.5)-- (6.5,2.);
		\draw (6.5,1.5)-- (6.,1.5);
		\draw (1.325,0.5) node[anchor=north west] {$\cK_2$};
		\draw (5.825,0.5) node[anchor=north west] {$\cK_3$};
		\draw (3.325,0.575) node[anchor=north west] {$\ol\cK_2$};
		\draw (7.825,0.575) node[anchor=north west] {$\ol\cK_3$};
		\draw [line width=2.8pt] (3.,2.)-- (3.,1.);
		\draw [line width=2.8pt] (4.,2.)-- (4.,1.);
		\begin{scriptsize}
		\draw [fill=black] (1.,2.) circle (4.5pt);
		\draw [fill=black] (1.,1.) circle (4.5pt);
		\draw [fill=black] (2.,2.) circle (4.5pt);
		\draw [fill=black] (2.,1.) circle (4.5pt);
		\draw [fill=black] (1.5,1.5) circle (4.5pt);
		\draw [fill=black] (5.5,2.) circle (1.5pt);
		\draw [fill=black] (5.5,1.5) circle (1.5pt);
		\draw [fill=black] (6.,2.) circle (1.5pt);
		\draw [fill=black] (6.,1.5) circle (1.5pt);
		\draw [fill=black] (6.5,2.) circle (1.5pt);
		\draw [fill=black] (6.5,1.5) circle (1.5pt);
		\draw [fill=black] (6.,0.75) circle (4.5pt);
		\draw [fill=ffffff] (3.,2.) circle (4.5pt);
		\draw [fill=ffffff] (4.,2.) circle (4.5pt);
		\draw [fill=ffffff] (3.,1.) circle (4.5pt);
		\draw [fill=ffffff] (4.,1.) circle (4.5pt);
		\draw [fill=black] (7.5,2.) circle (1.5pt);
		\draw [fill=black] (7.5,1.5) circle (1.5pt);
		\draw [fill=black] (8.,2.) circle (1.5pt);
		\draw [fill=black] (8.,1.5) circle (1.5pt);
		\draw [fill=black] (8.5,2.) circle (1.5pt);
		\draw [fill=black] (8.5,1.5) circle (1.5pt);
		\draw [fill=ffffff] (3.5,1.5) circle (4.5pt);
		\draw [fill=ffffff] (8.,0.75) circle (4.5pt);
		\end{scriptsize}
		\end{tikzpicture}
		\caption{The graph classes $\cal{K}_2,$ $\cal{K}_3$ and their complementary classes}\label{F:K2K3}
	\end{figure}
\end{itemize}
\end{theorem}

The classes of graphs $\cal{K}_2$ and $\cal{K}_3$, their complements $\ol{\cal{K}}_2$ and $\ol{\cal{K}}_3$, and subsequent graph classes below are to be interpreted in the following way. A small dot indicates a single vertex, a dark circle indicates a clique, while a white circle indicates a stable set; edges are indicated by thin lines, while a thick line between two spheres or between two sets of vertices implies that every vertex in one set is adjacent to every vertex in the other set.

The graphs listed in Figure \ref{F:D1toD6} are referred to as the (complex) 3-blocks; see \cite{AHMR88} where this terminology originates from. They are exactly the graphs with (complex) sparsity order 3, for which each induced subgraph has strictly smaller sparsity order.

The analogous result for $\BB{F}=\BB{R}$, also proved by Laurent \cite[Theorem 9]{L01}, is even more complicated. In the variation on (ii), besides the cycles $C_n,$ $n\geq 5$, another 15 graphs (\cite[Figure 2]{L01}) are excluded as induced subgraphs, while the analogue on (iii) states that $G$ is the clique-sum of graphs from 4 different classes (\cite[Figure 1]{L01}), namely the classes $\cal{G}_1$, $\cal{G}_2$, $\cal{G}_3$ and $\cal{G}_4$ depicted in Figure \ref{F:G1toG4} below (note that $\cal{G}_4=\cal{K}_2$).

	\begin{figure}[ht]
		\centering
		\begin{tikzpicture}[scale=1.25,line cap=round,line join=round,>=triangle 45,x=1.0cm,y=1.0cm]
		\clip(-5.,0.15) rectangle (4.5,5.565);
		\draw [line width=2.8pt] (1.,2.)-- (1.5,1.5);
		\draw [line width=2.8pt] (1.5,1.5)-- (1.,1.);
		\draw [line width=2.8pt] (1.,2.)-- (1.,1.);
		\draw [line width=2.8pt] (1.,1.)-- (2.,1.);
		\draw [line width=2.8pt] (2.,1.)-- (1.5,1.5);
		\draw [line width=2.8pt] (2.,1.)-- (2.,2.);
		\draw [line width=2.8pt] (1.5,1.5)-- (2.,2.);
		\draw [line width=2.8pt] (2.,2.)-- (1.,2.);
		\draw [line width=2.8pt] (3.,2.)-- (3.,1.);
		\draw [line width=2.8pt] (4.,2.)-- (4.,1.);
		\draw (-0.5,2.)-- (-0.5,1.5);
		\draw (-1.,2.)-- (-1.,1.5);
		\draw (-1.5,2.)-- (-1.5,1.5);
		\draw (-2.,2.)-- (-2.,1.5);
		\draw (-4.5,2.)-- (-4.,2.);
		\draw (-4.,2.)-- (-3.5,2.);
		\draw (-3.5,2.)-- (-3.,2.);
		\draw (-3.,1.5)-- (-3.5,1.5);
		\draw (-3.5,1.5)-- (-4.,1.5);
		\draw (-4.5,1.5)-- (-4.,1.5);
		\draw (-4.5,2.)-- (-4.,1.5);
		\draw (-4.,1.5)-- (-3.5,2.);
		\draw (-3.5,2.)-- (-3.,1.5);
		\draw (-3.,1.5)-- (-4.5,2.);
		\draw (-4.5,2.)-- (-3.5,1.5);
		\draw (-3.5,1.5)-- (-3.,2.);
		\draw (-3.,2.)-- (-4.5,1.5);
		\draw (-4.5,1.5)-- (-4.,2.);
		\draw (-4.,2.)-- (-3.5,1.5);
		\draw (-3.,2.)-- (-4.,1.5);
		\draw (-3.5,2.)-- (-4.5,1.5);
		\draw (-4.,2.)-- (-3.,1.5);
		\draw [shift={(-4.,1.241666666666666)}] plot[domain=0.9878827378391628:2.1537099157506305,variable=\t]({1.*0.9083333333333339*cos(\t r)+0.*0.9083333333333339*sin(\t r)},{0.*0.9083333333333339*cos(\t r)+1.*0.9083333333333339*sin(\t r)});
		\draw [shift={(-3.5,1.241666666666666)}] plot[domain=0.9878827378391628:2.1537099157506305,variable=\t]({1.*0.9083333333333339*cos(\t r)+0.*0.9083333333333339*sin(\t r)},{0.*0.9083333333333339*cos(\t r)+1.*0.9083333333333339*sin(\t r)});
		\draw [shift={(-3.75,1.)}] plot[domain=0.9272952180016122:2.214297435588181,variable=\t]({1.*1.25*cos(\t r)+0.*1.25*sin(\t r)},{0.*1.25*cos(\t r)+1.*1.25*sin(\t r)});
		\draw [shift={(-4.,2.2583333333333346)}] plot[domain=4.129475391428956:5.295302569340423,variable=\t]({1.*0.9083333333333343*cos(\t r)+0.*0.9083333333333343*sin(\t r)},{0.*0.9083333333333343*cos(\t r)+1.*0.9083333333333343*sin(\t r)});
		\draw [shift={(-3.5,2.2583333333333346)}] plot[domain=4.129475391428956:5.295302569340423,variable=\t]({1.*0.9083333333333343*cos(\t r)+0.*0.9083333333333343*sin(\t r)},{0.*0.9083333333333343*cos(\t r)+1.*0.9083333333333343*sin(\t r)});
		\draw [shift={(-3.75,2.5)}] plot[domain=4.068887871591405:5.355890089177974,variable=\t]({1.*1.25*cos(\t r)+0.*1.25*sin(\t r)},{0.*1.25*cos(\t r)+1.*1.25*sin(\t r)});
		\draw [rotate around={0.:(-3.75,1.75)},line width=1.2pt] (-3.75,1.75) ellipse (1.cm and 0.6614378277661477cm);
		\draw [line width=2.8pt] (-3.75,0.75)-- (-3.75,1.0885621722338523);
		\draw [rotate around={0.:(-2.5,4.5)},line width=1.2pt] (-2.5,4.5) ellipse (0.85cm and 0.4cm);
		\draw [line width=2.8pt] (-2.5,5.5)-- (-2.5,4.9);
		\draw [line width=2.8pt] (-2.5,3.5)-- (-2.5,4.1);
		\draw [line width=2.8pt] (1.5,4.)-- (1.,3.5);
		\draw [line width=2.8pt] (1.5,4.)-- (2.,3.5);
		\draw [line width=2.8pt] (3.,3.5)-- (4.,3.5);
		\draw [rotate around={0.:(1.5,3.65)},line width=1.2pt] (1.5,3.65) ellipse (0.75cm and 0.5590169943749502cm);
		\draw (1.25,5.25)-- (1.25,4.75);
		\draw (1.25,4.75)-- (1.75,4.75);
		\draw (1.75,4.75)-- (1.75,5.25);
		\draw (1.75,5.25)-- (1.25,5.25);
		\draw (3.25,5.25)-- (3.25,4.75);
		\draw (3.75,5.25)-- (3.75,4.75);
		\draw [line width=1.2pt] (1.5,5.) circle (0.45cm);
		\draw [line width=2.8pt] (1.5,4.55)-- (1.5,4.2090169943749505);
		\draw (-2.65,3.) node[anchor=north west] {$\cal{G}_1$};
		\draw (1.3,3.) node[anchor=north west] {$\cal{G}_2$};
		\draw (3.3,3.05) node[anchor=north west] {$\ol{\cal{G}}_2$};
		\draw (-3.375,4.725) node[anchor=north west] {\small chordal graph};
		\draw (-3.95,0.5) node[anchor=north west] {$\cal{G}_3$};
		\draw (-1.45,0.55) node[anchor=north west] {$\ol{\cal{G}}_3$};
		\draw (1.3,0.5) node[anchor=north west] {$\cal{G}_4$};
		\draw (3.3,0.55) node[anchor=north west] {$\ol{\cal{G}}_4$};
		\begin{scriptsize}
		\draw [fill=black] (1.,2.) circle (4.5pt);
		\draw [fill=black] (1.,1.) circle (4.5pt);
		\draw [fill=black] (2.,2.) circle (4.5pt);
		\draw [fill=black] (2.,1.) circle (4.5pt);
		\draw [fill=black] (1.5,1.5) circle (4.5pt);
		\draw [fill=black] (5.5,1.75) circle (1.5pt);
		\draw [fill=ffffff] (3.,2.) circle (4.5pt);
		\draw [fill=ffffff] (4.,2.) circle (4.5pt);
		\draw [fill=ffffff] (3.,1.) circle (4.5pt);
		\draw [fill=ffffff] (4.,1.) circle (4.5pt);
		\draw [fill=ffffff] (3.5,1.5) circle (4.5pt);
		\draw [fill=black] (-0.5,2.) circle (1.5pt);
		\draw [fill=black] (-0.5,1.5) circle (1.5pt);
		\draw [fill=black] (-1.,2.) circle (1.5pt);
		\draw [fill=black] (-1.,1.5) circle (1.5pt);
		\draw [fill=black] (-1.5,2.) circle (1.5pt);
		\draw [fill=ffffff] (-1.25,0.725) circle (4.5pt);
		\draw [fill=black] (-2.,2.) circle (1.5pt);
		\draw [fill=black] (-2.,1.5) circle (1.5pt);
		\draw [fill=black] (-3.,2.) circle (1.5pt);
		\draw [fill=black] (-3.,1.5) circle (1.5pt);
		\draw [fill=black] (-3.5,2.) circle (1.5pt);
		\draw [fill=black] (-3.5,1.5) circle (1.5pt);
		\draw [fill=black] (-4.,2.) circle (1.5pt);
		\draw [fill=black] (-4.,1.5) circle (1.5pt);
		\draw [fill=black] (-4.5,1.5) circle (1.5pt);
		\draw [fill=black] (-4.5,2.) circle (1.5pt);
		\draw [fill=black] (-3.75,0.75) circle (4.5pt);
		\draw [fill=black] (-2.5,3.5) circle (2.0pt);
		\draw [fill=black] (-2.5,5.5) circle (2.0pt);
		\draw [fill=black] (1.,3.5) circle (4.5pt);
		\draw [fill=black] (1.5,4.) circle (4.5pt);
		\draw [fill=black] (2.,3.5) circle (4.5pt);
		\draw [fill=ffffff] (3.,3.5) circle (4.5pt);
		\draw [fill=ffffff] (3.5,4.) circle (4.5pt);
		\draw [fill=ffffff] (4.,3.5) circle (4.5pt);
		\draw [fill=black] (1.25,5.25) circle (1.5pt);
		\draw [fill=black] (1.75,5.25) circle (1.5pt);
		\draw [fill=black] (1.25,4.75) circle (1.5pt);
		\draw [fill=black] (1.75,4.75) circle (1.5pt);
		\draw [fill=black] (3.25,5.25) circle (1.5pt);
		\draw [fill=black] (3.25,4.75) circle (1.5pt);
		\draw [fill=black] (3.75,5.25) circle (1.5pt);
		\draw [fill=black] (3.75,4.75) circle (1.5pt);
		\draw [fill=black] (-1.5,1.500339624539241) circle (1.5pt);
		\end{scriptsize}
		\end{tikzpicture}\caption{The graph classes $\cal{G}_1-\cal{G}_4$ and their complementary classes}\label{F:G1toG4}
	\end{figure}
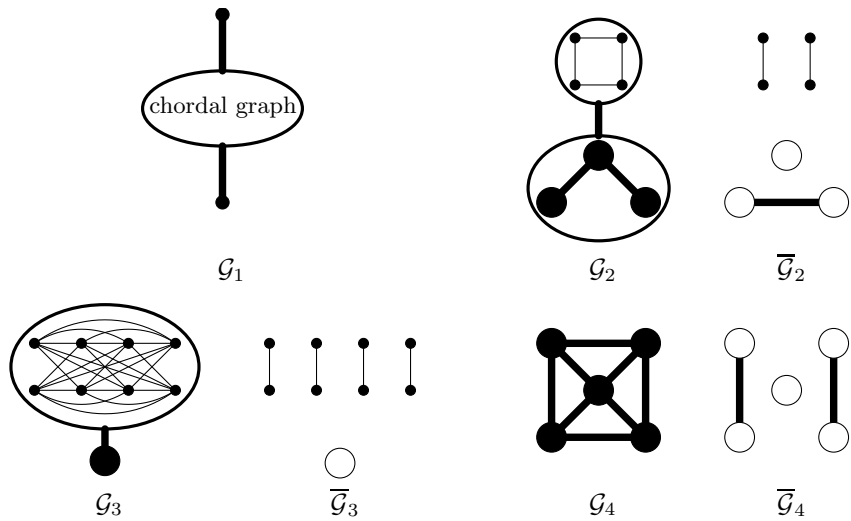
	
In \cite{L01}, Theorem \ref{T:SO2complex} is derived from the result for the case $\BB{F}=\BB{R}$, which in turn relies on a long series of technical lemmas. In the present paper we prove Theorem \ref{T:SO2complex} directly, without going through the case $\BB{F}=\BB{R}$. An important feature of our approach is the observation that the classes $\cal{K}_2$ and $\cal{K}_3$ can be seen as special cases of the larger class of graphs $\cal{F}=\cup_{k=0}^\infty \cal{F}_k$, where $\cal{F}_k$ is the class of graphs depicted in \ref{F:F_k}.

\begin{figure}[ht]
	\centering
	\begin{tikzpicture}[scale=1.55,line cap=round,line join=round,>=triangle 45,x=1.0cm,y=1.0cm]
	\clip(0.,-0.05) rectangle (8.,3.5);
	\draw [fill=black,fill opacity=1.0] (1.,2.5) circle (0.2cm);
	\draw [fill=black,fill opacity=1.0] (1.5,2.5) circle (0.2cm);
	\draw [fill=black,fill opacity=1.0] (3.,2.5) circle (0.2cm);
	\draw [fill=black,fill opacity=1.0] (3.5,2.5) circle (0.2cm);
	\draw [fill=black,fill opacity=1.0] (1.,1.5) circle (0.2cm);
	\draw [fill=black,fill opacity=1.0] (1.5,1.5) circle (0.2cm);
	\draw [fill=black,fill opacity=1.0] (3.,1.5) circle (0.2cm);
	\draw [fill=black,fill opacity=1.0] (3.5,1.5) circle (0.2cm);
	\draw [fill=black,fill opacity=1.0] (2.25,0.5) circle (0.2cm);
	\draw [line width=2.8pt] (1.,2.5)-- (1.5,1.5);
	\draw [line width=2.8pt] (1.5,2.5)-- (1.,1.5);
	\draw [line width=2.8pt] (1.,1.5)-- (1.5,1.5);
	\draw [line width=2.8pt] (1.,2.5)-- (1.5,2.5);
	\draw [line width=2.8pt] (3.,2.5)-- (3.5,1.5);
	\draw [line width=2.8pt] (3.,1.5)-- (3.5,2.5);
	\draw [line width=2.8pt] (3.5,2.5)-- (3.,2.5);
	\draw [line width=2.8pt] (3.,1.5)-- (3.5,1.5);
	\draw [shift={(2.25,-0.2625)},line width=2.8pt]  plot[domain=1.3056932622015485:1.8358993913882449,variable=\t]({1.*2.8625*cos(\t r)+0.*2.8625*sin(\t r)},{0.*2.8625*cos(\t r)+1.*2.8625*sin(\t r)});
	\draw [shift={(2.25,4.2625)},line width=2.8pt]  plot[domain=4.447285915791341:4.977492044978038,variable=\t]({1.*2.8625*cos(\t r)+0.*2.8625*sin(\t r)},{0.*2.8625*cos(\t r)+1.*2.8625*sin(\t r)});
	\draw [shift={(2.5,3.25)},line width=2.8pt]  plot[domain=4.193242866138167:5.231535094631212,variable=\t]({1.*2.0155644370746373*cos(\t r)+0.*2.0155644370746373*sin(\t r)},{0.*2.0155644370746373*cos(\t r)+1.*2.0155644370746373*sin(\t r)});
	\draw [shift={(2.5,0.75)},line width=2.8pt]  plot[domain=1.0516502125483738:2.0899424410414196,variable=\t]({1.*2.0155644370746373*cos(\t r)+0.*2.0155644370746373*sin(\t r)},{0.*2.0155644370746373*cos(\t r)+1.*2.0155644370746373*sin(\t r)});
	\draw [shift={(2.,0.75)},line width=2.8pt]  plot[domain=1.0516502125483738:2.0899424410414196,variable=\t]({1.*2.0155644370746373*cos(\t r)+0.*2.0155644370746373*sin(\t r)},{0.*2.0155644370746373*cos(\t r)+1.*2.0155644370746373*sin(\t r)});
	\draw [shift={(2.,3.25)},line width=2.8pt]  plot[domain=4.193242866138167:5.231535094631212,variable=\t]({1.*2.0155644370746373*cos(\t r)+0.*2.0155644370746373*sin(\t r)},{0.*2.0155644370746373*cos(\t r)+1.*2.0155644370746373*sin(\t r)});
	\draw [shift={(2.25,0.7468750000000003)},line width=2.8pt]  plot[domain=0.9513904377099842:2.190202215879809,variable=\t]({1.*2.1531249999999997*cos(\t r)+0.*2.1531249999999997*sin(\t r)},{0.*2.1531249999999997*cos(\t r)+1.*2.1531249999999997*sin(\t r)});
	\draw [shift={(2.25,3.2531250000000003)},line width=2.4pt]  plot[domain=4.092983091299777:5.3317948694696025,variable=\t]({1.*2.153125*cos(\t r)+0.*2.153125*sin(\t r)},{0.*2.153125*cos(\t r)+1.*2.153125*sin(\t r)});
	\draw [rotate around={0.:(2.25,2.)},line width=1.2pt] (2.25,2.) ellipse (2.018333527101401cm and 1.005569603071603cm);
	\draw [line width=2.8pt] (2.25,0.5)-- (2.244802330860033,0.9944337313019598);
	\draw [line width=2.8pt] (1.5,2.5)-- (3.,1.5);
	\draw [line width=2.8pt] (1.5,1.5)-- (3.,2.5);
	\draw [line width=2.8pt] (1.,1.5)-- (3.,2.5);
	\draw [line width=2.8pt] (1.,1.5)-- (3.5,2.5);
	\draw [line width=2.8pt] (3.,1.5)-- (1.,2.5);
	\draw [line width=2.8pt] (3.5,1.5)-- (1.5,2.5);
	\draw [line width=2.8pt] (3.5,1.5)-- (1.,2.5);
	\draw [line width=2.8pt] (1.5,1.5)-- (3.5,2.5);	
	\draw [color=ffffff](0.75,2.625) node[anchor=north west] {\fontsize{7}{1} $X_{\!1}$};
	\draw [color=ffffff](1.25,2.625) node[anchor=north west] {\fontsize{7}{1} $X_{\!2}$};
	\draw [color=ffffff](2.7,2.625) node[anchor=north west] {\fontsize{7}{1} $X_{\!k\text{-}1}$};
	\draw [color=ffffff](3.25,2.625) node[anchor=north west] {\fontsize{7}{1} $X_{\!k}$};	
	\draw [color=ffffff](0.775,1.625) node[anchor=north west] {\fontsize{7}{1} $Y_{\!1}$};
	\draw [color=ffffff](1.275,1.625) node[anchor=north west] {\fontsize{7}{1} $Y_{\!2}$};
	\draw [color=ffffff](2.725,1.625) node[anchor=north west] {\fontsize{7}{1} $Y_{\!k\text{-}1}$};
	\draw [color=ffffff](3.275,1.625) node[anchor=north west] {\fontsize{7}{1} $Y_{\!k}$};
	\draw [color=ffffff](2.035,0.625) node[anchor=north west] {\fontsize{7}{1} $Z$};
	\draw (2.035,0.25) node[anchor=north west] {${\cal{F}}_k$};	
	\draw [color=black](4.75,2.625) node[anchor=north west] {\fontsize{7}{1} ${X}_{\!1}$};
	\draw [color=black](5.25,2.625) node[anchor=north west] {\fontsize{7}{1} ${X}_{\!2}$};
	\draw [color=black](6.7,2.625) node[anchor=north west] {\fontsize{7}{1} ${X}_{\!k\text{-}1}$};
	\draw [color=black](7.25,2.625) node[anchor=north west] {\fontsize{7}{1} ${X}_{\!k}$};	
	\draw [color=black](4.775,1.625) node[anchor=north west] {\fontsize{7}{1} ${Y}_{\!1}$};
	\draw [color=black](5.275,1.625) node[anchor=north west] {\fontsize{7}{1} ${Y}_{\!2}$};
	\draw [color=black](6.7,1.625) node[anchor=north west] {\fontsize{7}{1} ${Y}_{\!k\text{-}1}$};
	\draw [color=black](7.275,1.625) node[anchor=north west] {\fontsize{7}{1} ${Y}_{\!k}$};
	\draw [color=black](6.035,0.625) node[anchor=north west] {\fontsize{7}{1} ${Z}$};
	\draw (6.035,0.285) node[anchor=north west] {$\ol{\cal{F}}_k$};
	\draw(5.,2.5) circle (0.2cm);
	\draw(5.5,2.5) circle (0.2cm);
	\draw(5.,1.5) circle (0.2cm);
	\draw(5.5,1.5) circle (0.2cm);
	\draw(7.,2.5) circle (0.2cm);
	\draw(7.5,2.5) circle (0.2cm);
	\draw(7.,1.5) circle (0.2cm);
	\draw(7.5,1.5) circle (0.2cm);
	\draw(6.25,0.5) circle (0.2cm);
	\draw [line width=2.8pt] (4.9999776835294245,2.300000001245062)-- (5.,1.7);
	\draw [line width=2.8pt] (5.5,2.3)-- (5.5,1.7);
	\draw [line width=2.8pt] (7.,1.7)-- (7.,2.3);
	\draw [line width=2.8pt] (7.5,1.7)-- (7.5,2.3);	
	\begin{scriptsize}
	\draw [fill=black] (6.,2.) circle (0.5pt);
	\draw [fill=black] (6.5,2.) circle (0.5pt);
	\draw [fill=black] (6.25,2.) circle (0.5pt);
	\draw [fill=black] (2.,2.) circle (0.5pt);
	\draw [fill=black] (2.5,2.) circle (0.5pt);
	\draw [fill=black] (2.25,2.) circle (0.5pt);
	\draw [fill=black] (2.,2.5) circle (0.5pt);
	\draw [fill=black] (2.5,2.5) circle (0.5pt);
	\draw [fill=black] (2.,1.5) circle (0.5pt);
	\draw [fill=black] (2.5,1.5) circle (0.5pt);
	\draw [fill=black] (2.25,2.5) circle (0.5pt);
	\draw [fill=black] (2.25,1.5) circle (0.5pt);
	\draw [fill=black] (6.,2.5) circle (0.5pt);
	\draw [fill=black] (6.5,2.5) circle (0.5pt);
	\draw [fill=black] (6.,1.5) circle (0.5pt);
	\draw [fill=black] (6.5,1.5) circle (0.5pt);
	\draw [fill=black] (6.25,1.5) circle (0.5pt);
	\draw [fill=black] (6.25,2.5) circle (0.5pt);
	\end{scriptsize}
	\end{tikzpicture}\caption{The graph classes $\cal{F}_k$ and $\ol{\cal{F}}_k$}\label{F:F_k}
\end{figure}
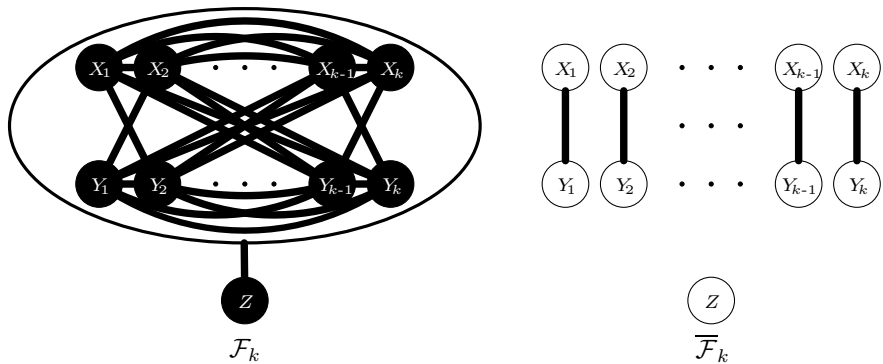

To guarantee that $\cal{F}_k$ and $\cal{F}_l$ are disjoint whenever $k\neq l$, we demand that the cliques $X_j\neq\emptyset$ and $Y_j\neq\emptyset,$ $j=1,\ldots, k$, but $Z=\emptyset$ is allowed.

Note that we have $\cal{K}_2=\cal{F}_2$ and $\cal{K}_3\subset \cal{F}_3$. Furthermore, the classes $\cal{G}_2$ and $\cal{G}_3$ that appear in the characterization of the real case are also subclasses of $\cal{F}$, namely, $\cal{G}_2\subset \cal{F}_3$ and $\cal{G}_3\subset \cal{F}_4$. Thus, only the class $\cal{G}_1$ differs in the sense that it is not a subclass of $\cal{F}$.\medskip

\paragraph{\bf Graph structure theorems}
Elements of $\ol{\cal{F}}$ are disjoint unions of complete bipartite graphs $K_{l,m}$, where the single stable set at the bottom should be viewed as $K_{l,0}$. Hence graphs in $\cal{F}$ can be described as the complements of disjoint unions of complete bipartite graphs. We now provide a different characterization of the class $\cal{F}$ in terms of forbidden induced subgraphs.

\begin{theorem}\label{T:graphstruc1}
Let $G$ be a simple undirected graph. Then $G\in\cal{F}$ if and only if $G$ does not contain  $P_4$ or $\ol{K}_3$ as induced subgraphs, i.e., $G$ contains no induced path of length 3 nor a stable set of cardinality 3.
\end{theorem}

Note that $P_4$ and $\ol{K}_3$ are elements of the class $\cal{G}_1$. Hence, as claimed above, $\cal{G}_1$ is indeed not a subclass of $\cal{F}$. Furthermore, the fact that $P_4$ is excluded as induced subgraph implies that the cycles $C_n,$ $n\geq 5,$ are also excluded.

There are many relevant classes of graphs that are characterized by excluding a set $\cal{L}$ of graphs as induced subgraphs, the so-called $\cL$-free graphs. See Chapter 7 in \cite{BLS99} for an overview of several of such classes and further references. Of particular interest to our class $\cal{F}$ is the class of cographs, graphs whose modular decomposition tree contains only parallel and series nodes; \cite[Section 11.3]{BLS99}. One characterization of cographs is that they are $P_4$-free. Hence, all graphs in $\cal{F}$ are cographs, and consequently, are non-prime, have the clique-kernel intersection property and every connected subgraph has diameter at most 2. More results can be found in Sections 1.5 and 11.3 of \cite{BLS99} and the references given there. Many of these results can easily be understood from the graphical representation in Figure \ref{F:F_k}.

A type of graph that is included in $\cal{F}$ is the cocktail party graph, cf., \cite{GMW86}.  Each class $\cal{F}_k$ contains exactly one cocktail party graph, the graph in $\cal{F}_k$ with the minimal number of vertices, namely one where the `single' clique is empty and all the `paired' cliques are singletons (forming a pair at the cocktail party). If one does allow the `single' clique (which may be interpreted as the collection of singles attending the cocktail party), these graphs can be characterized as the simple undirected graph that contain no induced $P_4$, $\ol{K}_3$ or $\ol{P}_3$; see the last part of the proof of Theorem \ref{T:SO2complex}, (ii) $\Leftrightarrow$ (iii) in Section \ref{S:grapgstruc} to see how excluding $\ol{P}_3$ compresses the `paired' cliques to singletons. See \cite{CLMS15,BM13,BM09} for some recent developments.

For the purposes of the present paper, however, we are interested in clique-sums of elements of $\cal{F}$. Note that $P_4$ and $\ol{K}_3$ can appear as induced subgraphs of such clique-sums, since $P_4$ and $\ol{K}_3$ themselves have clique cut-sets. This is why the characterization of the graphs that are clique-sums of elements of $\cal{F}$ involves different graphs, namely the graphs $C_n,$ $n\geq 5,$ and $D_1-D_4$.

\begin{theorem}\label{T:graphstruc2}
Let $G$ be a simple undirected graph. Then $G$ is a clique-sum of graphs from $\cal{F}$ if and only if $G$ does not have any of the graphs $C_n,$ $n\geq 5,$ and $D_1-D_4$ from Theorem \ref{T:SO2complex} as an induced subgraph.
\end{theorem}

We prove Theorems \ref{T:graphstruc1} and \ref{T:graphstruc2} in Section \ref{S:grapgstruc}. Once these results are established, the equivalence of (ii) and (iii) in Theorem \ref{T:SO2complex} readily follows; it is exactly when the graphs $D_5$ and $D_6$ in Figure \ref{F:D1toD6} are excluded as induced subgraphs, that a graph from the graph class $\cF$ must belong to $\cal{K}_2\cup\cal{K}_3$.\medskip

\paragraph{\bf Complex sparsity order.}
The second part of the paper (Section \ref{S:sparord}) has more of a linear algebra flavour. Here we give a brief sketch of the implication (i) $\Rightarrow$ (ii) of Theorem \ref{T:SO2complex} and a proof of the implication (iii) $\Rightarrow$ (i) of Theorem \ref{T:SO2complex}. The latter proof is also valid for the case $\BF=\BR$. Our aim here is to approach the result more from a matrix analysis perspective and to avoid usage of the so-called Dimension Theorem (Theorem 3.1 and Corollary 3.1 in \cite{AHMR88}), also referred to as the Frame Theorem in \cite{McM88}. The Dimension Theorem provides an easy way to determine if a matrix in $\PSD_G$ is extremal, but usage of this result sometimes conceals the argumentation.

Part of our proof is based on a solution to a seemingly independent question in linear algebra. Let $\BF^p$ have two orthogonal sum decompositions
\[
\BF^p=\cX_1\oplus\cX_2 \ands \BF^p=\cY_1\oplus\cY_2.
\]
We seek a subspace $\cZ\subset \BF^p$ of dimension two that splits over both orthogonal decompositions, that is, $\cZ= (\cZ\cap\cX_1)\oplus (\cZ\cap\cX_2)=(\cZ\cap\cY_1)\oplus (\cZ\cap\cY_2)$. It turns out that such a subspace $\cZ$ can always be constructed, see Proposition \ref{P:dim2subs} below, however, the construction we present is not trivial.\medskip

\paragraph{\bf Notation and terminology}
We conclude this introduction with a brief discussion of some of the notation and terminology used throughout the paper.

Throughout, we use $C_n$ and $P_n$ to denote the (chordless) cycle and (chordless) path with $n$ vertices. The complete graph with $n$ vertices is denoted by $K_n$ and we write $K_{l,m}$ for the complete bipartite graph of $l+m$ vertices with a bipartite partitioning of the vertices in subsets of $l$ and $m$ vertices.

Let $G=(V,E)$ be a simple undirected graph. Then $\ol{G}$ denotes the complement of $G$. For a subset $S\subset V$, we write $G[S]$ for the subgraph of $G$  induced by $S$, that is $G[S]=(S,E_S)$ with $E_S=E\cap (S\times S)$. We say that a graph $G'$ is an induced subgraph of $G$ whenever there exists a $S\subset V$ such that $G'$ is isomorphic to $G[S]$. A subset $S$ of $V$ is called a clique if $G[S]$ is complete and a stable set if $\ol{G}[S]$ is complete. If $G$ is connected, then a clique $S$ in $G$ is called a clique cut-set in $G$ if $G[E\backslash S]$ is no longer connected. If two graphs $G_i=(V_i,E_i)$, $i=1,2$, both have a clique of the same cardinality, then the clique-sum of $G_1$ and $G_2$ is the graph $G=(V_1\cup V_2, E_1\cup E_2)$ where the clique vertices in $V_1$ and $V_2$ are identified. Clique-sums of more than two graphs are defined inductively.

\section{Graph structure theorems}\label{S:grapgstruc}

In this section we prove the two structure theorems stated in the introduction, i.e., Theorems \ref{T:graphstruc1} and \ref{T:graphstruc2}. We also prove the equivalence of statements (ii) and (iii) in Theorem \ref{T:SO2complex}. We start with a proof of Theorem \ref{T:graphstruc1}.

\begin{proof}[\bf Proof of Theorem \ref{T:graphstruc1}.]
	Clearly, if $G\in\cal{F}_k$ it can not contain $P_4$ or $\ol{K}_3$ as an induced subgraph.

	Conversely, assume $G$ does not contain $P_4$ or $\ol{K}_3$ as an induced subgraph. If $G$ is a complete graph, then $G\in \cal{F}_0$ and we are done. Thus assume $G$ is not a complete graph. Then there exist $x_1,y_1\in V$ which are non-adjacent. We now define the sets
	\[
	X_1=\lbrace{v\in V: v\neq y_1, (v,y_1)\notin E}\rbrace
	\quad\text{ and }\quad
	Y_1=\lbrace{u\in V:u\neq x_1, (u,x_1)\notin E}\rbrace.
	\]
	Then $x_1\in X_1$ and $y_1\in Y_1$. We claim that $X_1$ and $Y_1$ are disjoint cliques, with no edges between them and that every vertex in $X_1\cup Y_1$ is adjacent to every vertex in $V\setminus[X_1\cup Y_1]$.

	Indeed, if $X_1$ is not a clique, say $x,x'\in X_1$ are non-adjacent, then $\lbrace{x,x',y_1}\rbrace$ is a stable set of cardinality 3, a contradiction. Thus $X_1$ is a clique. Similarly we obtain that $Y_1$ is a clique. If $X_1\cap Y_1\neq\emptyset$, say $v\in X_1\cap Y_1$, then $\lbrace{v,x_1,y_1}\rbrace$ is a stable set of cardinality 3, again a contradiction. Thus $X_1\cap Y_1=\emptyset$. Next, assume that there exists some vertex $x\in X_1$ adjacent to a vertex $y\in Y_1.$ Note that $x\neq x_1$ and $y\neq y_1,$ per the definition of $X_1$ and $Y_1.$ Furthermore, $x$ is adjacent to $x_1$ and $y$ is adjacent to $y_1,$ since $X_1$ and $Y_1$ are cliques. Then $\lbrace{x_1,x,y,y_1}\rbrace$ induces a $P_4,$ a contradiction. Thus, there are no edges between $X_1$ and $Y_1.$

	If $V\setminus [X_1\cup Y_1]$ is empty, we are done. If not let $z\in V\setminus[X_1\cup Y_1],$ then $z$ is adjacent to $x_1$ and $y_1$, else $z\in X_1\cup Y_1$. Suppose now that $u\in X_1\cup Y_1$ is non-adjacent to $z$. Then $u\neq x_1,y_1$ and $u$ is adjacent to exactly one of $x_1$ and $y_1$, since $X_1$ and $Y_1$ are disjoint cliques. Thus $\lbrace{u,x_1,y_1,z}\rbrace$ induces a $P_4$ in $G$, a contradiction. This proves our claim.

	Set $V_1=V\setminus[X_1\cup Y_1]$. If $V_1$ is a clique in $G$, then $G\in\cal{F}_1$ with $Z=V_1$ and we are done. Else, set $G_1=G[V_1]$. Then $G_1$ does not contain $P_4$ or $\ol{K}_3$ as an induced subgraph. We now repeat the process described above with $G_1$, until we are left with a (possibly empty) clique; in case the process is repeated $k$ times before the clique appears, we find that $G\in \cal{F}_k$.
\end{proof}

The proof of Theorem \ref{T:graphstruc2} is more involved. We start with an important observation. Different from $P_4$ and $\ol{K}_3$, the graphs $D_1-D_4$ have no clique cut-set. Therefore, for $j=1,\ldots,4$, if $D_j$  appears as an induced subgraph of $G$, and $G$ has a clique cut-set, then $D_j$ must appear in one of the subgraphs from the clique cut-set decomposition of $G$. This has two implications:
\begin{itemize}
	\item The necessity claim of Theorem \ref{T:graphstruc2} follows directly from the easy observation that $D_1,\ldots,D_4\notin \cal{F}$.

	\item For the sufficiency claim of Theorem \ref{T:graphstruc2} we may assume without loss of generality that $G$ has no clique cut-set. In particular, we may assume $G$ is connected, since the null graph is a clique cut-set for the connected parts of a graph.

\end{itemize}

Hence it remains to prove the sufficiency claim of Theorem \ref{T:graphstruc2}. In the remainder of this section we shall assume $G$ has no clique cut-set. We shall also assume $G$ is not a complete graph, since $\cal{F}_0$ is precisely the set of complete graphs, and so, if $G$ is complete the theorem clearly holds.

\begin{proposition}\label{P:part}
	Let $G=(V,E)$ be a simple undirected graph. If $G$ is not a complete graph and has no clique cut-sets, then $V$ can be partitioned as $$V=X_0\cup V_0\cup Y_0,$$ with $X_0$, $V_0$ and $Y_0$ non-empty, $G[X_0]$ connected, no edges between vertices in $X_0$ and $Y_0$, the set $V_0$ is not a clique, every vertex in $Y_0$ is adjacent to every vertex in $V_0$ and every vertex in $V_0$ is adjacent to, at least, one vertex in $X_0$.
\end{proposition}

\begin{proof}[\bf Proof.]
	For $X\subset V$, define
	\[
	V_X=\{v\in V\setminus X \colon (v,x)\in E  \mbox{ for some }x\in X\}\ands
	Y_X=V\setminus (X\cup V_X).
	\]
	Then $V=X\cup V_X\cup Y_X$ is a partitioning of $V$. Now chose $X_0\subset V$ to be a subset of $V$ of maximal cardinality so that $G[X_0]$ is connected and $Y_0:=Y_{X_0}\neq\emptyset$. Since $G$ is not complete, there exists an $x\in V$ which is not adjacent to all other vertices in $V$. Hence for $X=\{x\}$ we have $Y_X\neq\emptyset$ and $G[X]$ is connected. This shows that $X_0\neq\emptyset$. By construction, there can be no vertices between edges in $X_0$ and $Y_0$ and any $v_0\in V_0:=V_{X_0}$ must be adjacent to a vertex in $X_0$. The graph $G$ has no clique cut-sets, and hence is connected. Therefore, also $V_0\neq \emptyset$. If $V_0$ were a clique, then it would be a clique cut-set, which we assume not to exist.

	It remains to show that every vertex in $Y_0$ is adjacent to every vertex in $V_0$. If this were not the case, then there would be a $v_0\in V_0$ which is not adjacent to all vertices in $Y_0$. However, then we could add $v_0$ to $X_0$, contradicting the maximality assumption.
\end{proof}

For the previous result it was not needed to exclude the graphs $C_n,$ $n\geq 5,$ or $D_1-D_4$ as induced subgraphs. Excluding these graphs leads to the following proposition, which in turn allows us to apply Theorem \ref{T:graphstruc1} to $G[V_0],$ since the proposition shows that the graph induced by $V_0$ contains no $\ol{K}_3$ or $P_4.$

\begin{proposition}\label{P:Struc1App}
	Let $G=(V,E)$ be a simple undirected graph which is not a complete graph and has no clique cut-sets. Assume that the cycles $C_n$, $n\geq 5$, and the graphs $D_1-D_4$ are not induced subgraphs of $G$. Let $V=X_0\cup V_0\cup Y_0$ be a partitioning as in Proposition \ref{P:part}. Then $\ol{K}_3$ and $P_4$ are not induced subgraphs of $G[V_0]$.
\end{proposition}

Before proving this result we will need the following lemma and its corollary. The lemma and its corollary do not only play an important role in the proof of Proposition \ref{P:Struc1App}, but also in several results in the remainder of this section. These two results also appear in \cite{L01}, as Claims 3 and 5, except that Claim 5 in \cite{L01} also includes a uniqueness claim, for which additional graphs have to be excluded. Proofs are added for the sake of completeness.

\begin{lemma}\label{L:commonvert}
	Assume that the cycles $C_n$, $n\geq 5$, are not induced subgraphs of $G$. For any non-adjacent vertices $v,w\in V_0$ there exists a $x\in X_0$ adjacent to both $v$ and $w$.
\end{lemma}

\begin{proof}[\bf Proof.]
	Let $v,w\in V_0$ such that there exists no vertex in $X_0$ adjacent to both $v$ and $w$. Let $P$ be a shortest path between $v$ and $w$ whose intermediate vertices are in $X_0$. Since $v$ and $w$ are adjacent to vertices in $X_0$ and $G[X_0]$ is connected, such a path exists, and it must consist of at least four vertices. Choose $y\in Y_0$ arbitrarily. Then $y$ is adjacent to $v$ and $w$, but not to any other vertex in $P$. Thus, adding $y$ we find an induced cycle $C_n$ of length $n\geq 5$, unless when $v$ and $w$ are adjacent. We conclude that if $v,w\in V_0$ are non-adjacent, then there must be a $x\in X_0$ adjacent to both $v$ and $w$.
\end{proof}

\begin{corollary}\label{C:commonvert}
	If the cycles $C_n$, $k\geq 5$, are not induced subgraphs of $G$ and $\{v,w,z\}\subset V_0$ is a stable set in $G$, then there exists a vertex $x\in X_0$ adjacent to $v$, $w$ and $z$.
\end{corollary}

\begin{proof}[\bf Proof.]
	Assume there is no vertex in $X_0$ adjacent to $v$, $w$ and $z$. By Lemma \ref{L:commonvert}, there exist vertices $x,s,t\in X_0$ with $x$ adjacent to $v$ and $w$, $s$ adjacent to $v$ and $z$, and $t$ adjacent to $w$ and $z$, and by our assumption $x$, $s$ and $t$ must be distinct. Moreover, $x$, $s$ and $t$ are pairwise non-adjacent, since if this were not the case, we would obtain an induced $C_5$ on one of the subsets $\{y,v,x,s,z\}$, $\{y,w,s,t,v\}$ or $\{y,w,x,t,z\}$, with $y\in Y_1$ arbitrary. However, with $x$, $s$ and $t$ pairwise non-adjacent the subset $\{x,w,s,z,t,v\}\subset V$ induces a $C_6$, a contradiction. Thus there must be a vertex in $X_0$ adjacent to $v$, $w$ and $z$.
\end{proof}

\begin{proof}[\bf Proof of Proposition \ref{P:Struc1App}.]
	Assume $\{v,w,z\}$ is a stable set in $G[V_0]$. By Corollary \ref{C:commonvert}, there exists a $x_0\in X_0$ adjacent to $v$, $w$ and $z$. Then $\{v,w,z,x_0,y\}$, with $y\in Y_0$ arbitrary, induces a $D_4$ in $G$. Hence $G[V_0]$ cannot contain a stable set of size 3.

	Assume that $[v,w,r,s]$ is an induced $P_4$ in $G[V_0]$. Then $(v,s)\notin E$. By Lemma \ref{L:commonvert}, there exists a $x_0\in X_0$ adjacent to $v$ and $s$. Let $y\in Y_0$. Depending on whether $x_0$ is adjacent to both, one or none of $w$ and $r$, we obtain an induced $D_3$, $D_1$ or $C_5$ (on $\{v,w,r,s,x_0\}$, since $\ol{C_5}=C_5$), respectively; see Figure \ref{F:ForbGraph1}.
	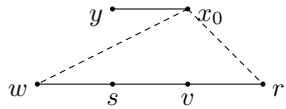
\begin{figure}[ht]
		\centering
		\begin{tikzpicture}[scale=0.5,line cap=round,line join=round,>=triangle 45,x=1.0cm,y=1.0cm]
		\clip(0.,0.) rectangle (8.,3.5);
		\draw (1.,1.)-- (3.,1.);
		\draw (3.,1.)-- (5.,1.);
		\draw (5.,1.)-- (7.,1.);
		\draw (3.,3.)-- (5.,3.);
		\draw [dash pattern=on 2pt off 2pt] (5.,3.)-- (7.,1.);
        \draw [dash pattern=on 2pt off 2pt] (5.,3.)-- (1.,1.);
		\draw (5.,3.24) node[anchor=north west] {$x_0$};
		\draw (3.,3.24) node[anchor=north east] {$y$};
		\draw (1.,1.24) node[anchor=north east] {$w$};
		\draw (3.,1.) node[anchor=north] {$s$};
		\draw (5.,1.) node[anchor=north] {$v$};
		\draw (7.,1.24) node[anchor=north west] {$r$};
		\begin{scriptsize}
		\draw [fill=black] (1.,1.) circle (1.5pt);
		\draw [fill=black] (3.,1.) circle (1.5pt);
		\draw [fill=black] (5.,1.) circle (1.5pt);
		\draw [fill=black] (7.,1.) circle (1.5pt);
		\draw [fill=black] (3.,3.) circle (1.5pt);
		\draw [fill=black] (5.,3.) circle (1.5pt);
		\end{scriptsize}
		\end{tikzpicture}\caption{The complementary graph on the set $\{v,w,r,s,x_0,y\}$}\label{F:ForbGraph1}
	\end{figure}

	\noindent
	Each case leads to a contradiction, hence we obtain that $G[V_0]$ cannot contain $P_4$ as an induced subgraph.
\end{proof}

In light of Proposition \ref{P:Struc1App} we may apply Theorem \ref{T:graphstruc1} to obtain $G[V_0]\in\cal{F}_k$ for some $k\geq 1$; $k=0$ cannot occur, since $V_0$ is not a clique. Denote the `paired' cliques in Figure \ref{F:F_k} by $X_i, Y_i,$ for $i=1,\ldots,k,$ and the `single' clique by $Z$. The structure of $G[V_0]$ and the interaction between the vertices in $Y_0$ and $V_0$ and vertices in $X_0$ and $Y_0$ is now understood. It remains to show that $Y_0$ is a clique and that $X_0$ is a clique all of whose vertices are adjacent to all vertices in $V_0$. The former is now easy to prove.

\begin{lemma}\label{L:Y0clique}
	Assume a simple undirected graph $G$ is not a complete graph, has no clique cut-sets, and does not contain, as induced subgraphs, any of the graphs $C_n$, $n\geq 5$, or $D_1-D_4.$ Then $Y_0$ from the partitioning of Proposition \ref{P:part} is a clique.
\end{lemma}

\begin{proof}[\bf Proof.]
	Assume $y_0,y_0'\in Y_0$ are non-adjacent. Let $x_1\in X_1$, $y_1\in Y_1$ and let $x_0\in X_0$ be adjacent to $x_1$ and $y_1$, which is possible according to Lemma \ref{L:commonvert}. Then $y_0$ and $y_0'$ are adjacent to both $x_1$ and $y_1$ and $y_0$, $y_0'$ and $x_0$ are pairwise non-adjacent. This shows that $\{x_1,y_1,x_0,y_0,y_0'\}$ induces a $D_4$ in $G$, a contradiction. Hence all vertices in $Y_0$ are adjacent, as claimed.
\end{proof}

To determine the structure of $X_0$ and its interaction with $V_0$ we first consider the subset
\begin{equation}\label{X0til}
	\wtil{X}_0=\{x_0\in X_0\colon (x_0,v)\in E\mbox{ for every }v\in X_j\cup Y_j,\mbox{ for }j=1,\ldots,k\},
\end{equation}
which we later show to be equal to $X_0$.

\begin{lemma}\label{L:Xtil0Alt}
	The set $\wtil{X}_0$ is also given by
\begin{align}\label{X0hat}
	\wtil{X}_0&=\nonumber\\ & \{x_0\in X_0\colon\mbox{$(x_0,x_j),(x_0,y_j)\in E$ for some $x_j\in X_j$ and some $y_j\in Y_j$}\}.
\end{align}
\end{lemma}

\begin{proof}[\bf Proof.]
	Write $\what{X}_0$ for the right hand side of \eqref{X0hat}. Clearly, $\wtil{X}_0\subset\what{X}_0$. Now assume $x_0\in\what{X}_0$. Without loss of generality, $j=1$, that is, there exist $x_1\in X_1$ and $y_1\in Y_1$ adjacent to $x_0$.

	We first show $x_0$ is adjacent to all vertices in $X_1\cup Y_1$. Assume this is not the case, say $(x_0,y_1')\notin E$ for a $y_1'\in Y_1$; the case $(x_0,x_1')\notin E$ for a $x_1'\in X_1$ goes analogously. Since $(x_1,y_1')\notin E$, by Lemma \ref{L:commonvert}, there exists a $x_0'\in X_0$ adjacent to $x_1$ and $y_1'$. Let $y_0\in Y_0$. Since the structure of $Y_0$ and $X_1$ and $Y_1$ and their interaction is known, the only remaining possible adjacencies are between $x_0'$ and $x_0$ and between $x_0'$ and $y_1$. Below is a table of the various adjacency cases of $x_0',$ a check in the appropriate box means that $x_0'$ is adjacent to the relevant vertex in $G.$ Figure \ref{F:ForbGraph2} gives a sketch of the subgraph of $\ol{G}$ induced by $\{y_0,x_0,x_0',x_1,y_1,y_1'\}$; for the last case, note that the complement of a $C_5$ is again a $C_5$.
	
	\begin{center}
		\begin{tabular}{|c|c|c|c|}
			\hline
			$x_0$ & $y_1$ & Vertex set & Induced subgraph\\
			\hline
			$\checkmark$ & $\checkmark$ & $\{y_0,x_0,y_1',x_1,y_1',x_0'\}$ & $D_1$\\
			\hline
			$\checkmark$ & & $\{y_0,x_0,y_1',x_1,y_1',x_0'\}$ & $D_2$\\
			\hline
			 & $\checkmark$ & $\{y_0,x_0,x_0',x_1,y_1'\}$ & $D_4$\\
			\hline
			 & & $\{x_0',x_0,y_1',x_1,y_1\}$ & $C_5$\\
			\hline
		\end{tabular}
	\end{center}
	
	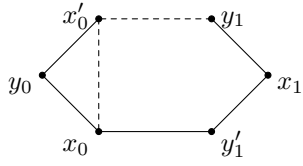
\begin{figure}[ht]
		\centering
		\begin{tikzpicture}[scale=0.75,line cap=round,line join=round,>=triangle 45,x=1.0cm,y=1.0cm]
		\clip(0.25,0.5) rectangle (5.75,3.5);
		\draw (1.,2.)-- (2.,1.);
		\draw (1.,2.)-- (2.,3.);
		\draw (2.,1.)-- (4.,1.);
		\draw (4.,1.)-- (5.,2.);
		\draw (5.,2.)-- (4.,3.);
		\draw [dash pattern=on 2pt off 2pt] (2.,3.)-- (2.,1.);
		\draw [dash pattern=on 2pt off 2pt] (4.,3.)-- (2.,3.);
		\draw (3.9999111751869036,3.3) node[anchor=north west] {$y_1$};
		\draw (2.0011422650473185,3.45) node[anchor=north east] {$x^\prime_0$};
		\draw (1.0017578099775257,2.12787799626425) node[anchor=north east] {$y_0$};
		\draw (2.0011422650473185,1.028493541194458) node[anchor=north east] {$x_0$};
		\draw (3.9999111751869036,1.178493541194458) node[anchor=north west] {$y^\prime_1$};
		\draw (4.999295630256697,2.12787799626425) node[anchor=north west] {$x_1$};
		\begin{scriptsize}
		\draw [fill=black] (1.,2.) circle (1.5pt);
		\draw [fill=black] (2.,1.) circle (1.5pt);
		\draw [fill=black] (2.,3.) circle (1.5pt);
		\draw [fill=black] (4.,1.) circle (1.5pt);
		\draw [fill=black] (5.,2.) circle (1.5pt);
		\draw [fill=black] (4.,3.) circle (1.5pt);
		\end{scriptsize}
		\end{tikzpicture}
		\caption{The complementary graph on the set $\{y_0,x_0,y_1',x_1,y_1',x_0'\}$}\label{F:ForbGraph2}		
	\end{figure}
	
	We see that each case leads to a contradiction, hence $x_0$ is adjacent to all vertices in $X_1\cup Y_1$.

	Next we show $x_0$ is adjacent to all vertices in $X_j\cup Y_j$ for $j>1$. Assume this is not the case for a $j>1$. Without loss of generality there exists a $x_j\in X_j$ non-adjacent to $x_0$. Choose $y_j\in Y_j$ and $y_0\in Y_0$ arbitrarily. All adjacency relations are known apart from whether $x_0$ is adjacent to $y_j$, which gives an induced $D_3$, or whether $x_0$ and $y_j$ are non-adjacent, in which case we find $D_4$ as an induced subgraph; see Figure \ref{F:ForbGraph3}.
	
	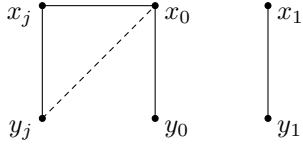
\begin{figure}[ht]
		\centering
		\begin{tikzpicture}[scale=0.75,line cap=round,line join=round,>=triangle 45,x=1.0cm,y=1.0cm]
		\clip(0.25,0.5) rectangle (5.75,3.5);
		\draw (1.,1.)-- (1.,3.);
		\draw (1.,3.)-- (3.,3.);
		\draw (3.,3.)-- (3.,1.);
		\draw (5.,3.)-- (5.,1.);
		\draw [dash pattern=on 2pt off 2pt] (1.,1.)-- (3.,3.);
		\draw (1.0017578099775257,3.1272624513340417) node[anchor=north east] {$x_j$};
		\draw (1.0017578099775257,1.1284935411944579) node[anchor=north east] {$y_j$};
		\draw (3.000526720117111,1.1284935411944579) node[anchor=north west] {$y_0$};
		\draw (3.000526720117111,3.1272624513340417) node[anchor=north west] {$x_0$};
		\draw (4.999295630256697,3.1272624513340417) node[anchor=north west] {$x_1$};
		\draw (4.999295630256697,1.1284935411944579) node[anchor=north west] {$y_1$};
		\begin{scriptsize}
		\draw [fill=black] (1.,1.) circle (1.5pt);
		\draw [fill=black] (1.,3.) circle (1.5pt);
		\draw [fill=black] (3.,3.) circle (1.5pt);
		\draw [fill=black] (3.,1.) circle (1.5pt);
		\draw [fill=black] (5.,3.) circle (1.5pt);
		\draw [fill=black] (5.,1.) circle (1.5pt);
		\end{scriptsize}
		\end{tikzpicture}
		\caption{The complementary graph on the set $\{x_1,y_1,x_j,y_j,x_0,y_0\}$}\label{F:ForbGraph3}
	\end{figure}
	
	Since all cases lead to a forbidden induced subgraph, we reach a contradiction. Hence $x_0$ is adjacent to all vertices in $X_j\cup Y_j$ for $j>1$ as well, and thus $x_0\in \wtil{X}_0$.
\end{proof}


\begin{proposition}\label{P:Xtil0props}
	The set $\wtil{X}_0$ is a non-empty clique in $G.$
\end{proposition}

\begin{proof}[\bf Proof]
	Since $\wtil{X}_0$ is also given by \eqref{X0hat} and, by Lemma \ref{L:commonvert}, for $x_1\in X_1$ and $y_1\in Y_1$ there exists a $x_0\in X_0$ adjacent to $x_1$ and $y_1$, we obtain that $\wtil{X}_0\neq\emptyset$.

	Assume $\wtil{X}_0$ is not a clique, say $x_0,x_0'\in \wtil{X}_0$ are non-adjacent. Let $x_1\in X_1$, $y_1\in Y_1$ and $Y_0\in Y_0$. Then $\{x_0,x_0',y_0\}$ forms a stable set (hence an induced $K_3$ in $\ol{G}$) whose vertices are all adjacent to both $x_1$ and $y_1$. Since $(x_1,y_1)\notin E$, we obtain that $\{x_0,x_0',y_0,x_1,y_1\}$ induces $D_4$ in $G$.
\end{proof}

\begin{proposition}\label{P:zadjtilx0}
	Every vertex in $\wtil{X}_0$ is adjacent to every vertex in $V_0$.
\end{proposition}

\begin{proof}[\bf{Proof}]
	By definition, any vertex in $\wtil{X}_0$ is adjacent to all vertices in $X_j\cup Y_j$ for $j=1,\ldots k$. Hence it remains to show they are adjacent to the vertices in $Z$ as well. Assume this is not the case, say $z\in Z$ and $x_0\in \wtil{X}_0$ are non-adjacent. Choose $y_0\in Y_0$, $x_1\in X_1$ and $y_1\in Y_1$ arbitrarily.
	
	Since $z\in V_0$, it must be adjacent to a vertex in $X_0$, see Proposition \ref{P:part}, and since $G[X_0]$ is connected there exists a path from $z$ to $x_0$ whose internal vertices are all in $X_0$. Clearly $j\geq 1$, for $x_0$ is not adjacent to $z$. Let $[z,x_0^{(1)},\ldots,x_0^{(j)},x_0]$ be a shortest path of this type. Now, $x_0^{(1)}$ cannot be in $\wtil{X}_0$. Indeed, if $x_0^{(1)}\in \wtil{X}_0$, then $z,x_0,x_0^{(1)}$ would all be adjacent to both $x_1$ and $y_1$, and $(x_0,x_0^{(1)})\in E$ since $\wtil{X}_0$ is a clique. Thus, with $x_0^{(1)}\in \wtil{X}_0$ we find $D_3$ as an induced subgraph on the vertex set $\lbrace{y_0,x_0,x_0',x_1,y_1,z}\rbrace.$ Since there is no shorter path from $z$ to $x_0$ none of the internal vertices can be adjacent to $z$, apart from $x_0^{(1)}$, or to $x_0$, apart from $x_0^{(j)}$. Without loss of generality $x_0^{(i)}\notin \wtil{X}_0$ for all $i$, otherwise change $x_0$ to $x_0^{(i)}\in \wtil{X}_0$.
	
	We show that $j=1$. Assume $j\geq 2$. Note that $[z,x_0^{(1)},\ldots,x_0^{(j)},x_0,x_1]$ and $[z,x_0^{(1)},\ldots,x_0^{(j)},x_0,y_1]$ are both cycles of length $j+3$ in $G$. If none of the vertices $x_0^{(i)}$ are adjacent to either $x_1$ and $y_1$, then we find $C_{3+j}$ as an induced subgraph. Hence both cycles must contain chords, that is, $x_1$ must be adjacent to an internal vertex and $y_1$ must be adjacent to an internal vertex. Let $x_0^{(i)}$ be the first vertex in the path adjacent to $x_1$. Then $[z,x_0^{(1)},\ldots,x_0^{(i)},x_1]$ is an induced cycle of length $i+2$. Hence $i\leq 2$. Thus, $x_1$ is adjacent to either $x_0^{(1)}$ or $x_0^{(2)}$. Similarly, $y_1$ is adjacent to either $x_0^{(1)}$ or $x_0^{(2)}$. However, $x_1$ and $y_1$ cannot be adjacent to the same internal vertex, since all $x_0^{(i)}$ are in $X_0\setminus \wtil{X}_0$; see Lemma \ref{L:Xtil0Alt}. Thus, without loss of generality, $(x_0^{(1)},x_1)\in E$ and $(x_0^{(2)},y_1)\in E$. If $j=2$, then it follows that $x_0^{(2)}$ is adjacent to $x_0$ and we find $D_2$ as an induced subgraph on the vertex set $\lbrace{x_0,x_0^{(1)},x_0^{(2)},x_1,y_1,z}\rbrace.$ If $j\geq 3$ it follows that $x_0^{(2)}$ is non-adjacent to $x_0$ and so $[x_0,x_1,x_0^{(1)},x_0^{(2)},y_1]$ is an induced cycle in $G$ of length 5. All possibilities lead to a contradiction, hence $j=1$.
	
	Since $j=1$, $x_0^{(1)}$ is adjacent to both $z$ and $x_0$ and to, at most, one of $x_1$ and $y_1$. If $x_0^{(1)}$ were non-adjacent to both $x_1$ and $y_1$, then we find $D_4$ as an induced subgraph on the vertex set $\{x_0^{(1)},x_1,y_1, z,x_0\}$. Say $(x_0^{(1)},x_1)\in E$. Note that $y_0$ is adjacent to $x_1,y_1,z$ and non-adjacent to $x_0, x_0^{(1)}$. Thus we find $D_1$ as an induced subgraph on the vertex set $\lbrace{x_1,y_1,x_0^{(1)},y_0,x_0,z}\rbrace$. Again a contradiction. Hence every vertex $x_0\in X_0$ must be adjacent to every vertex $z\in Z$.
\end{proof}

\begin{proof}[\bf Proof of Theorem \ref{T:graphstruc2}.]
	As observed above, the fact that $D_1-D_4$ contains no clique cut-sets and are not in $\cal{F}$ (easily checked by comparing the complements) proves the necessity part of Theorem \ref{T:graphstruc2}.

	It remains to show that excluding $D_1-D_4$ leads to $G$ being a clique-sum of graphs from $\cal{F}$. Again using the fact that $D_1-D_4$ contain no clique cut-sets, we may assume $G$ has no clique cut-set, and in that case we need to show $G\in\cal{F}$. Let $V=X_0\cup Y_0\cup V_0$ be the partitioning of $V$ from Proposition \ref{P:part}. Applying Proposition \ref{P:Struc1App}, Theorem \ref{T:graphstruc1} and Lemma \ref{L:Y0clique} all that remains is to prove that $X_0$ is a clique, each vertex of which is adjacent to every vertex in $V_0$. The subset $\wtil{X}_0\subset X_0$ defined in \eqref{X0til} turns out to have exactly these properties; see Proposition \ref{P:Xtil0props}. In other words, if we set $T=X_0\setminus \wtil{X}_0$, then $G[V\setminus T]\in \cal{F}_{k+1}$, if $G[V_0]\in\cal{F}_k$. Hence, if $u,v\in V\setminus T$ are non-adjacent, then there must be a $0\leq j\leq k$ such that $u\in X_j$ and $v\in Y_j$ (or conversely), except for $j=0$ where $X_0$ should be replaced by $\wtil{X}_0$. To complete the proof, we show $\wtil{X}_0=X_0$, i.e., $T=\emptyset$.

	Assume $T\neq\emptyset$. We claim that for any connected subset $S\subset T$, the set $A(S):=\{w\in V\setminus T\colon (s,w)\in E\mbox{ for a }s\in S\}$ is a clique. To see that this is the case, it suffices to consider the case that the vertices in $S$ form a path in $T$. So, say $S=\{t_0,\ldots,t_l\}$ with $[t_0,\ldots,t_l]$ a path in $T$. Assume $A(S)$ is not a clique. Then there exists a $j$ such that $x_j\in X_j$, $y_j\in Y_j$ with $x_j,y_j\in A(S)$. It cannot occur that $x_j$ and $y_j$ are both adjacent to a vertex $t_i$ in the path, since Lemma \ref{L:Xtil0Alt} would then imply $t_i\in\wtil{X}_0$. Without loss of generality, we may assume $x_j$ is adjacent to $t_0$ and $y_j$ is adjacent to $t_l$ and $x_j$ and $y_j$ are non-adjacent to all $t_i$ for $i=1,\ldots,l-1$, otherwise, simply take a sub-path with this property. Choose $y_0\in Y_0$ arbitrary. The set $\{t_0,\ldots,t_l,y_j,y_0,x_j\}$ then induces a cycle $C_{4+l}$ in $G$. This contradicts our assumptions, unless we have that $l=0$. However, if $l=0$, $S$ consists of a single vertex, and, as seen above, $A(\{t_0\})$ is a clique, since otherwise, $t_0\in \wtil{X}_0$.

	Now let $S$ be any connected component in $T$ (possibly $T$ itself), then $S$ is non-empty, $A(S)$ is a clique, $S$ and $T\setminus S$ are disconnected and $S$ and $V\setminus (A(S)\cup T)$ are disconnected. Thus $A(S)$ is a clique cut-set in $G$, separating $S$ and $V\setminus (A(S)\cup S)$, a contradiction, since $G$ is assumed not to have clique cut-sets.
\end{proof}
	
	A graph from the class of clique-sums of elements in $\cal{F}$ with no clique cut-set must be in $\cal{F}$ itself. The next corollary then follows directly from Theorems \ref{T:graphstruc1} and \ref{T:graphstruc2}.

\begin{corollary}
	Let $G$ be a simple undirected graph which does not contain any of the cycles $C_n$, $n\geq 5$, or the graphs $D_1-D_4$ from Theorem \ref{T:SO2complex} as induced subgraphs. If $G$ contains no clique cut-set, then $P_4$ and $\ol{K}_3$ also cannot appear as induced subgraphs of $G$.
\end{corollary}

\begin{proof}[\bf Proof of Theorem \ref{T:SO2complex}, (ii) $\Leftrightarrow$ (iii).]
	One easily verifies that none of the cycles $C_n$, $n\geq 5$, or the graphs $D_1-D_6$ are contained in $\cal{K}_2$ or $\cal{K}_3$. Hence, they cannot appear as induced subgraphs of graphs in $\cal{K}_2$ or $\cal{K}_3$. Since none of $C_n$, $n\geq 5$, and $D_1-D_6$ have a clique cut-set, they also cannot appear as induced subgraphs of clique-sums of graphs from $\cal{K}_2$ or $\cal{K}_3$. This proves that (iii) implies (ii).

	Conversely, assume $G$ is a graph that does not contain $C_n$, $n\geq 5$, and $D_1-D_6$ as induced subgraphs. By Theorem \ref{T:graphstruc2}, $G$ is a clique-sum of graphs from the graph class $\cal{F}$. Since $C_n$, $n\geq 5$, and $D_1-D_6$ have no clique cut-sets, all graphs in any clique-sum decomposition of $G$ also cannot contain $C_n$, $n\geq 5$, and $D_1-D_6$ as induced subgraphs. Hence, without loss of generality, we may assume $G$ has no clique cut-sets, and therefore $G\in\cal{F}$. Since $D_6\in\cal{F}_j$ for $j\geq 4$, we have $G\in\cal{F}_2\cup \cal{F}_3$ (note that graphs in $\cal{F}_1$ have a clique cut-set, so $G\notin\cal{F}_1$). Whenever any of the cliques in the clique pairs of $\cal{F}_3$ consists of more than one vertex, $D_5$ appears as an induced subgraph. Consequently, $G$ will either be in $\cal{F}_2=\cal{K}_2$ or in $\cal{F}_3$ with single vertex cliques in the clique pairs, that is, in $\cal{K}_3$. Thus, (ii) implies (iii).
\end{proof}

\section{The complex sparsity order}\label{S:sparord}

In this section we give a brief sketch of the proof of the implication (i) $\Rightarrow$ (ii) of Theorem \ref{T:SO2complex}, mostly for completeness, and a proof of the remaining implication (iii) $\Rightarrow$ (i), thus completing the proof of Theorem \ref{T:SO2complex}. The known proofs of the implication (iii) $\Rightarrow$ (i) all rely strongly on the so-called Dimension Theorem \cite[Corollary 3.1]{AHMR88}. Here we present a more elementary proof of (iii) $\Rightarrow$ (i) that does not require this result, and which holds for both $\BF=\BC$ and $\BF=\BR$. Note that (iii) $\Rightarrow$ (i) indeed holds for $\BF=\BR$ as well, but that this is not the case for (i) $\Rightarrow$ (ii). We start with some known general observations regarding the sparsity order of a graph that we will need in the sequel. Throughout this section $\BF$ will be equal to $\BC$ or $\BR$; we use notation and terminology corresponding to $\BF=\BC$ which should be properly interpreted if $\BF=\BR$.

\begin{lemma}\label{L:OrdF}
Let $G$ be a simple graph. Then the following holds:
\begin{itemize}

\item[(a)] Let $H$ be an induced subgraph of $G$. Then $\ord_\BF(H)\leq \ord_\BF(G)$.

\item[(b)] Let $\wtil{G}$ be the graph obtained by adding a clique in such a way that each vertex in the clique is adjacent to each vertex of $G$. Then $\ord_\BF(\wtil{G})= \ord_\BF(G)$.

\item[(c)] If $G$ is the clique-sum of graphs $G_1,\ldots,G_p$, then
\[
\ord_\BF(G)=\max_{1\leq i\leq p}(\ord_\BF(G_i)).
\]
\end{itemize}
\end{lemma}

\begin{proof}[\bf Proof.]
The claims of (a) and (c) can be found in \cite[Theorem 4.1]{AHMR88} and \cite[Theorem 3.1]{HPR89}, respectively.

We prove (b). Clearly $\ord_\BF(\wtil{G})\geq\ord_\BF(G)$, since any extremal element $X\in \PSD_G$ gives an extremal element $X'=\sbm{X&0\\0&0}\in \PSD_{\wtil{G}}$ of the same rank. Here the zero-blocks correspond to the added vertices.

To show the reverse inequality, let $X'\in \PSD_{\wtil{G}}$. Then $X'$ has the block form
\[
X'=\mat{cc}{X&B^*\\B&Y}\mbox{ with } X\in\PSD_G,
\]
with no other restrictions on $B$ and $Y$ than that $X'$ must be positive semidefinite. If $X=0$, then $B=0$ and $X'$ is extremal if and only if $Y$ is extremal in the matrix cone generated by the added clique, which occurs if and only if $\rank Y=1$. Hence we may assume $X\neq 0$, which we will do in the remainder of the proof.

We show that $X'$ extremal in $\PSD_{\wtil{G}}$ implies $X$ is extremal in $\PSD_G$ and $\rank X =\rank X'$. In particular, for each extremal element in $\PSD_{\wtil{G}}$ there is an extremal element in $\PSD_G$ of the same rank, and hence, $\ord_\BF(\wtil{G})\leq\ord_\BF(G)$, which completes the proof. Now assume $X'$ is extremal in $\PSD_{\wtil{G}}$. Since $X'\geq 0$, we can factor $X'$ as $X'=W^*W$ with $W=\sbm{W_1&W_2}$ full row rank, i.e., $W$ maps onto $\BF^p$ with $p=\rank X'$; the block decomposition of $W$ corresponds to the decomposition of $X'$ above.

Clearly $\rank X \leq\rank X'$. Assume $\rank X <\rank X'$. Then $\rank W_1=\rank X<\rank X'=p$, and hence $\im W_1\neq \BF^p$. Take any $v\in\BF^p$ with $\|v\|=1$ and $v\perp \im W_1$. Set $X_1'=W^*(I-vv^*)W$ and $X_2'=W^*vv^*W$. Then $X'=X_1'+X_2'$ and $X_1',X_2'\geq 0$. Furthermore, since $W_1^*v=0$, we have
\begin{align*}
X_1&=\mat{c}{W_1^*\\W_2^*} (I-vv^*) \mat{cc}{W_1 & W_2}=\mat{cc}{X&B\\B&W_2^* (I-vv^*) W_2},\\
X_2&=\mat{c}{W_1^*\\W_2^*} vv^* \mat{cc}{W_1 & W_2}=\mat{cc}{0&0\\0&W_2^* vv^* W_2}.
\end{align*}
Both are in $\PSD_{\wtil{G}}$, and not a positive scalar multiple of $X'$, since $X\neq0$. This is in contradiction with the extremality of $X'$, hence $\rank X=\rank X'$.

Now assume $X$ is non-extremal, say $X=X_1+X_2$ with $X_i\in \PSD_G$, $X_i\neq\la X$ for a $\la>0$. Since $X,X_1,X_2\geq 0$ the identity $X=X_1+X_2$ implies that $\kr X\subset \kr X_i$, hence $\im X_i \subset \im X$. Then also $\im X_i^\half \subset \im X^\half=\im W_1^*$. By Douglas' Lemma \cite{D66}, $X_i^\half= W_1^* T_i$ for a $p\times p$ matrix $T_i$. Hence $X_i=X_i^\half X_i^\half=W_1^* R_i W_1$ for $R_i=T_iT_i^*$. Now $R_i\neq \la I$ for all $\la>0$, for otherwise $X_i=\la X$. Also, since $W_1^* W_1=X=X_1+X_2=W_1^*(R_1+R_2)W_1$ we may assume $R_1+R_2=I$. Now set $X'_i=W^* R_i W$, $i=1,2$. Then $X'=X_1'+X_2',$ $X_1',X_2'\geq 0$ and $X_i'\neq\la X,$ since $R_i\neq \la I.$ Moreover, $X_i'$ has the form $X_i'=\sbm{X_i & *\\ *& *}$. Since $X_i\in \PSD_G$ we see that $X_i'\in\PSD_{\wtil{G}}$. Again we reach a contradiction with the extremality of $X'$, hence $X$ must be extremal.
\end{proof}

\begin{proof}[\bf Sketch of the proof of Theorem \ref{T:SO2complex} (i) $\Rightarrow$ (ii).]
By Lemma \ref{L:OrdF}(a) it follows that a graph with complex sparsity order less than or equal to 2 cannot contain, as an induced subgraph, a graph with sparsity order greater than 2. Thus, it suffices to show that the graphs $C_n,$ $n\geq 5,$ and $D_1-D_6$ have sparsity order greater than 2. In \cite[Lemma 2.8]{McM88} this is shown to be the case if $G$ is one of $C_n,$ $n\geq 5,$ and $D_1-D_4$ and in \cite[p.556]{L01} it is shown for $D_5$ and $D_6$. To prove these claims one constructs an element $X\in\PSD_G$ with $\rank(X)=p>2$ in the form $X=W^*W$. Let $w_1,\ldots, w_n$ be the columns of $W$ and define
\[
\cU:=\spn\{w_i w_j^*\colon (i,j)\in \overline{E}\}.
\]
A consequence of the Dimension Theorem is that $X$ is extreme in $\PSD_G$ if and only if $\dim_\BC(\cU)=k^2-1$. For $G$ one of the above graphs, this enables one to explicitly construct an extreme element $X$ in $\PSD_G$ with $\rank(X)>2$.
\end{proof}

The main implication of Lemma \ref{L:OrdF} (b) used in the present paper is that the unpaired clique in the graphs from the class $\cF$ is irrelevant for the sparsity order of the graph, hence we may assume it to be void, which we will do in the sequel. We start with a few additional reductions for extremal elements in $\PSD_G$ for $G\in\cF$.

Assume $X\in\PSD_G$ with $G\in \cF_k$ for some $k$ so that the unpaired clique is void and the sizes of the paired cliques are given by $p_{i,l}$ for $i=1,\ldots,k$ and $l=1,2$. Set $p_i=p_{i,1}+p_{i,2}$. Then $X$ has the following form
\begin{align}
& X=\mat{cccc}{M_1&A_{1,2}&\cdots&A_{1,k}\\ A_{1,2}^*&M_2&&\vdots\\ \vdots&&\ddots&A_{k-1,k}\\ A_{1,k}^*&\cdots&A_{k-1,k}^*&M_k}, \label{Xdec}\\
& M_i=\mat{cc}{M_{i,1}&0\\0& M_{i,2}},\ A_{i,j}=\mat{cc}{A_{i,j}^{1,1} & A_{i,j}^{1,2}\\[.1cm] A_{i,j}^{2,1}& A_{i,j}^{2,2}},\notag
\end{align}
where $M_{i,l}\in\BF^{p_{i,l}\times p_{i,l}}$ and $A_{i,j}^{l,m}\in\BF^{p_{i,l}\times p_{j,m}}$. Note that positive semidefiniteness is independent of the choice of basis. However, we have to preserve the structure imposed by $G$, in this case the zero blocks in $M_i$. This still allows us to change the basis for each of the $\BF^{k_{i,l}}$. Further, note that since $X$ is a positive semidefinite block matrix, $\kr M_{i,l}=\coker M_{i,l}$ is contained in $\kr A_{i,j}^{l,m}$ and $\coker (A_{i,j}^{l,m})^*$, for $j=1,\ldots,k$ and $m=1,2$.
Write $\wtil{M}_{i,l}$ for the compressions of ${M}_{i,l}$ to $\kr M_{i,l}^\perp$ and $\coker M_{i,l}^\perp$ and $\wtil{A}_{i,j}^{l,m}$ for the compressions of ${A}_{i,j}^{l,m}$ to $\kr M_{i,l}^\perp$ and $\coker M_{j,m}^\perp$, and define $\wtil{X}$ according to \eqref{Xdec} with ${M}_{i,l}$ and ${A}_{i,j}^{l,m}$ replaced by $\wtil{M}_{i,l}$ and $\wtil{A}_{i,j}^{l,m}$, respectively, writing $\wtil{M}_i$ and $\wtil{A}_{i,j}$ instead of $M_i$ and $A_{i,j}$.
Since $\kr \wtil{M}_i$ is the orthogonal direct sum of $\kr \wtil{M}_{i,1}$ and $\kr\wtil{M}_{i,1}$, the matrix $\wtil{M}_i$ is invertible.
Furthermore, $X$ is extremal if and only if $\wtil{X}$ is extremal in $\PSD_{G'}$ with $G'\in \cF_k$ obtained by reducing the paired cliques of $G$ in size accordingly to the size reduction from $X$ to $\wtil{X}$; the choice of basis for $\kr M_{i,l}^\perp$ is irrelevant in this regard. That $X$ is extremal if and only if $\wtil{X}$ is extremal is easily seen, since if $X$ can be written as $X=X_1+X_2$ for $X_1,X_2\in\PSD_G$, then the kernel $\kr M_{i,l}$, for $i=1,\ldots,k$ and $l=1,2$, must also be included in the kernels and co-kernels of the appropriate blocks of $X_1$ and $X_2$ in their block decompositions of the form \eqref{Xdec}. Hence, without loss of generality we may assume $M_i$ is positive definite. Now set $T=\diag(M_1^{-1/2},\ldots, M_k^{-1/2})$ and $\what{X}=T^* X T$. Then $\what{X}$ is in $\PSD_G$, and in it's block representation \eqref{Xdec} we have $M_{i,l}=I_{p_{i,l}}$, hence $M_i=I_{p_i}$. Moreover, the map $Y\mapsto T^* Y T$ is a bijective map from $\PSD_G$ onto $\PSD_G$, and as a result, $X$ is extremal if and only if $\what{X}$ is extremal. Hence we may assume that $M_i=I_{p_i}$ for each $i$, that is, $X$ takes the form
\begin{equation}\label{Xdec2}
X=\mat{cccc}{I_{p_1}&C_{1,2}&\cdots&C_{1,k}\\ C_{1,2}^*&I_{p_2}&&\vdots\\ \vdots&&\ddots&C_{k-1,k}\\ C_{1,k}^*&\cdots&C_{k-1,k}^*&I_{p_k}},
\end{equation}
where the fact that $X$ is positive implies that all $C_{i,j}$ must be contractions.

We will next specialize to the cases $\cK_2$ and $\cK_3$. The following intermezzo will be of use when dealing with the $\cK_2$ case.\medskip

\paragraph{\bf Intermezzo:\ Two dimensional subspaces with a 2-way splitting}$ $\\
Let $\BF^p$ have two orthogonal sum decompositions:
\begin{equation}\label{CpDec}
\BF^p=\cX_1\oplus\cX_2 \ands \BF^p=\cY_1\oplus\cY_2.
\end{equation}
We seek a subspace $\cZ\subset \BF^p$ of dimension two that splits over both orthogonal decompositions, that is,
\begin{equation}\label{2waysplit}
\cZ= (\cZ\cap\cX_1)\oplus (\cZ\cap\cX_2)=(\cZ\cap\cY_1)\oplus (\cZ\cap\cY_2).
\end{equation}
Clearly, we must have $p\geq2$. The following result shows that such a splitting always exists.

\begin{proposition}\label{P:dim2subs}
Let $p\geq2$. Assume that $\BB{F}^p$ has two orthogonal decompositions $\BB{F}^p=\cal{X}_1\oplus\cal{X}_2=\cal{Y}_1\oplus\cal{Y}_2$. Then there exists a two-dimensional linear space $\cal{Z}\subseteq\BB{F}^p$ that splits over both orthogonal decompositions, that is, \eqref{2waysplit} holds.
\end{proposition}

\begin{proof}[\bf Proof]
The proof is split into three parts.

\paragraph{\bf Part 1}
We first show that without loss of generality we may assume that $\cal{X}_i\cap\cal{Y}_j=\lbrace{0}\rbrace,$ for all $i,j$. To the contrary, assume $\cK:=\cal{X}_i\cap\cal{Y}_j\neq \lbrace{0}\rbrace$. Set $\cal{\what{X}}_1=\cal{X}_1\ominus\cal{K}$ and $\what{\cal{Y}}_1=\cal{Y}_1\ominus\cal{K}$.  Then $\cH:=\BB{F}^p\setminus\cal{K}=\what{\cal{X}}_1\oplus\cal{X}_2=\what{\cal{Y}}_1\oplus\cal{Y}_2$ and
$\what{\cal{X}}_i\cap\what{\cal{Y}}_j=\lbrace{0}\rbrace$. Note that a 2-dimensional subset of $\cH$ that splits over the two orthogonal decompositions of $\cH$ also splits over the two orthogonal decompositions of $\BF^p$. If $\dim(\cH)\geqslant 2$, then we consider the other intersections till we arrive at the case $\cal{X}_i\cap\cal{Y}_j=\lbrace{0}\rbrace,$ for all $i,j$. If $\dim(\cH)< 2$, then there are two cases to consider:
\begin{itemize}
\item[$(1)$] $\dim(\cal{K})\geqslant 2:$ Take any $\cal{Z}\subseteq\cal{K}$ with $\dim(\cal{Z})=2.$
\item[$(2)$] $\dim(\cal{K})=1:$ Then, necessarily, $p=2$ and we take $\cal{Z}=\BB{F}^p=\BB{F}^2.$
\end{itemize}
Consequently, we need but consider the case where $\cal{X}_i\cap\cal{Y}_j=\lbrace{0}\rbrace,$ for all $i,j$, which we will do in the sequel.

\paragraph{\bf Part 2}
Next we show that $p$ is even, $p_i:=\dim(\cX_i)=p/2$ and $q_i:=\dim(\cY_i)=p/2$ for $i=1,2$. Note that $p=p_1+p_2=q_1+q_2$. Hence there must exist $i,j$ so that $p_i\geq p/2$ and $q_j\geq p/2$. By assumption we have $\dim(\cal{X}_i\cap\cal{Y}_j)=0$ so that
\begin{align*}
p_i+q_j=\dim(\cal{X}_i)+\dim(\cal{Y}_j)&=\dim(\cal{X}_i\cap\cal{Y}_j)+\dim(\cal{X}_i+\cal{Y}_j)
=\dim(\cal{X}_i+\cal{Y}_j)\leqslant p.
\end{align*}
Hence $p\geq p_i+q_j\geq p/2+p/2=p$, so that $p_i+q_j=p$. Since both $p_i,q_j\geq p/2$, we must have $p_i=q_j=p/2$ and $p$ must be even. Also, $p-p_i=p/2=p-q_j$ such that $p_i=q_i=p/2$ for $i=1,2$.

\paragraph{\bf Part 3} In the final part we prove the claim. Set $d=p/2.$ After a change of coordinates we can assume that
$$\cal{X}_1=\spn\lbrace{e_1,\ldots,e_d}\rbrace \quad\text{ and }\quad \cal{X}_2=\spn\lbrace{e_{d+1},\ldots,e_{p}}\rbrace,$$
with $e_1,\ldots,e_p$ the standard basis for $\BF^p$. Let $\{f_1,\ldots,f_d\}$ and $\{f_{d+1},\ldots,f_p\}$ be bases for $\cY_1$ and $\cY_2$, respectively. For $i=1,2$ and $k=1,\ldots,p$ set $f_k^{(i)}=P_{\cX_i}f_k$. Define
\[
T=\mat{cc}{T_{1,1}&T_{1,2}\\ T_{2,1} & T_{2,2}}
=\mat{ccc|ccc}{
f_1^{(1)}&\cdots&f_d^{(1)}&f_{d+1}^{(1)}\cdots&f_p^{(1)}\\\hline
f_1^{(2)}&\cdots&f_d^{(2)}&f_{d+1}^{(2)}\cdots&f_p^{(2)}}\ands
T_i=\mat{c}{T_{1,i}\\T_{2,i}}.
\]
Note that each $T_{i,j}$ is a $d\times d$ matrix. Moreover, $T_i$ is a one-to-one map from $\cX_i$ onto $\cY_i$ for $i=1,2$ and hence the columns of $T_1$ are orthogonal to the columns of $T_2$. The latter is equivalent to $T_{2,2}^*T_{2,1}+T_{1,2}^*T_{1,1}=0$. We now show that $\kr T_{i,j}=\{0\}$ for each $i,j$, hence $T_{i,j}$ is invertible. Indeed, say $0\neq v\in \kr T_{1,j}$. Then $0\neq T_j v\in \cY_j$ and $T_j v=\sbm{0\\ T_{2,j}v}\in \cX_2$. Hence $\cY_j\cap \cX_2\neq \{0\}$, in contradiction with the assumption $\cY_j\cap \cX_i= \{0\}$ for all $i,j$. Similarly, a contradiction is obtained from $0\neq v\in \kr T_{2,j}$. In particular, $T_{1,1}$ and $T_{2,2}$ are invertible, so that we can factor $T$ as
\[
T=\mat{cc}{I_d &T_{1,2}T_{2,2}^{-1}\\ T_{2,1}T_{1,1}^{-1}& I_d} \mat{cc}{T_{1,1}&0\\ 0 & T_{2,2}}.
\]
Applying another change of basis to $\cX_1$ and $\cX_2$, using $T_{1,1}$ and $T_{2,2}$, respectively, we see that without loss of generality $T_{1,1}=I_d=T_{2,2}$. The orthogonality of the columns of $T_1$ and $T_2$ then reduces to $T_{2,1}+T_{1,2}^*=0$, hence $T_{2,1}=-T_{1,2}^*$. In other words, we may assume $T$ has the form
\[
T=\mat{cc}{I_d & U\\ -U^* & I_d}\quad\mbox{and hence}\quad
T_1=\mat{c}{I_d\\ -U^*},\quad T_2=\mat{c}{U\\ I_d},
\]
for some invertible $d\times d$ matrix $U$ (viewed as a one-to-one map from $\cX_2$ onto $\cX_1$).

Let $v\in \cX_2$ be an eigenvector of $U^*U$ with eigenvalue $\la$. Note that $U^*U$ is positive definite, so that $\la>0$. Set
$$
\cal{Z}=\spn\left\lbrace{\begin{pmatrix} Uv \\ 0\end{pmatrix},\begin{pmatrix}0 \\ v\end{pmatrix}}\right\rbrace.
$$
Then clearly $\dim(\cZ)=2$ and $\cZ$ splits over $\BF^p=\cX_1\oplus \cX_2$. Note that $\cZ$ can also be written as
$$\cal{Z}=\spn\left\lbrace{\begin{pmatrix}Uv \\ -\la v\end{pmatrix},\begin{pmatrix}Uv \\ v\end{pmatrix}}\right\rbrace
=\spn\left\{T_1Uv, T_2v\right\}.
$$
The latter identity shows that $\cZ$ also splits over $\BF^p=\cY_1\oplus \cY_2$.
\end{proof}

\paragraph{\bf The case $G\in \cK_2$}
Let $G\in\cK_2=\cF_2$. We show that any extremal element of $\PSD_G$ has rank of at most two. Note that in this case an $X\in \PSD_G$ may be assumed to be of the form
\begin{equation}\label{XdecF2}
X=\mat{cc}{I_{p_1} & C\\ C^* & I_{p_2}}
\quad\mbox{with}\quad
C=\mat{cc}{C_{1,1} & C_{1,2}\\ C_{2,1} & C_{2,2}},
\end{equation}
where the block decomposition of the contraction $C$ is according to $p_1=p_{1,1}+p_{1,2}$ and $p_{2}=p_{2,1}+p_{2,2}$. We may assume $p_{i,j}\neq 0$, for otherwise $X$ may be viewed as an element of $\PSD_{G'}$ for a $G'\in \cF_1$, for which we know $\ord_\BF(G')\leq1$.

\begin{lemma}\label{L:unitary}
Assume $X$ in \eqref{XdecF2} is extremal. Then $C$ is unitary. In particular, $\rank X=p_1=p_2$. Furthermore, $p$ is even and $p_{i,j}=p/2$ for $i,j=1,2$.
\end{lemma}

\begin{proof}[\bf{Proof}]
Factor $X=W^*W$ with $W=\mat{cc}{W_1 & W_2}$ and $W_i=\mat{cc}{W_{i,1} & W_{i,2}}$, where $W$ decomposes according to the decomposition of $X$ and $W_i$ according to the decomposition of $C$ in \eqref{XdecF2}. Then $W_i^* W_i=I_{p_i}$, hence $W_i$ is an isometry. Assume $W$ has full row rank, hence $W$ maps onto $\BF^p$ with $p=\rank X$. This implies $p_1,p_2\leq p$. Hence for both $i=1,2$ either $p_{i,1}\leq p/2$ or $p_{i,2}\leq p/2$. Assume there exists $i,j\in\{1,2\}$ so that $p_{i,j}<p/2$, which is always the case if $p$ is odd. Choose $i',j'\in\{1,2\}$ with $i\neq i'$ so that $p_{i',j'}\leq p/2$. Then
\[
\rank \mat{c}{W_{i,j}^*\\ W_{i',j'}^*}\leq p_{i,j}+p_{i',j'}<p, \quad \mbox{hence}\quad \kr \mat{c}{W_{i,j}^*\\ W_{i',j'}^*}\neq\{0\}.
\]
Take $v\in \BF^p$, $\|v\|=1$ so that $W_{i,j}^*v=0$ and $W_{i',j'}^*v=0$. Since $i\neq i'$, we obtain that $W_k^* vv^* W_{k}=\sbm{* &0 \\0 &*}$ for both $k=1,2$. Hence $X_1=W^*(I-vv^*)W$ and $X_2=W^*vv^* W$ are in $\PSD_G$. Moreover, $\rank X_2=1$, since $W$ has full row rank, hence $X_1$ and $X_2$ are no positive scalar multiples of $X$. This is in contradiction with the extremality of $X$. Hence $p_{i,j}=p/2$ for all $i,j\in\{1,2\}$. Then $p_1=p_2=p$. We already observed that $W_1$ and $W_2$ are isometries. Now, $p_1=p_2=p$ implies $W_1$ and $W_2$ are square, hence unitary. Then $C= W_1^*W_2$ will also be unitary.
 \end{proof}

Using Proposition \ref{P:dim2subs} and the fact the $C$ in \eqref{XdecF2} is unitary whenever $X$ is extremal in $\PSD_G$, we prove the following result.

\begin{proposition}\label{P:rank2}
Let $X$ in \eqref{XdecF2} be extremal. Then $\rank(X)\leq 2$.
\end{proposition}

\begin{proof}[\bf Proof.]
Since $X$ is extremal, $C$ in \eqref{XdecF2} is unitary and hence $X$ factors as
\[
X=\mat{c}{I_p\\ C^*}\mat{cc}{I_p & C}.
\]
Define
\[
J_1=\mat{cc}{I_{p_{1,1}}&0\\0& -I_{p_{1,2}}}\ands
J_2= C J_2' C^*\quad\mbox{with}\quad J_2'=\mat{cc}{I_{p_{2,1}}&0\\0& -I_{p_{2,2}}}.
\]
Then $J_1$ and $J_2$ are signature matrices of size $p\times p$. Let $\cZ$ be a subspace of $\BF^p$ that splits over both orthogonal sum decompositions of $\BF^p$ associated with $J_1$ and $J_2$, which exists by Proposition \ref{P:dim2subs}, and write $Q$ for the orthogonal projection on $\cZ$. Then $J_i Q J_i= Q$ for $i=1,2$. Since $J_2= C J_2' C^*$ and $C$ is unitary, this implies $J_2' C^* Q C J_2'=C^* Q C $. However, the identities $J_1 Q J_1= Q$ and $J_2' C^* Q C J_2'=C^* Q C$ imply that $Q$ and $C^* Q C$ have a block diagonal form with respect to the decompositions $p=p_{1,1}+p_{1,2}$ and $p=p_{2,1}+p_{2,2}$, respectively. Therefore
\[
X_1=\mat{c}{I_p\\ C^*}Q\mat{cc}{I_p & C}=\mat{cc}{Q& QC\\ C^*Q & C^*QC}
\]
and
\[
X_2=\mat{c}{I_p\\ C^*}(I-Q)\mat{cc}{I_p & C}
\]

are both in $\PSD_G$ and $X=X_1+X_2$. Since $X$ is extremal, $X_1$ is a positive scalar multiple of $X$. However, $\rank X_1=2$, so that $\rank X=2$.
\end{proof}

\paragraph{\bf The case $G\in \cK_3$}
Let $G\in\cK_3$. Assume $X$ is extremal in $\PSD_G$ and of the form \eqref{Xdec2} with $k=3$ and all entries of size $2\times 2$. Again, factor $X$ as $X=W^*W,$ where $\rank(X)=\rank(W)=p$ and decompose $W$ as
\[
W=\begin{bmatrix} W_1 & W_2 & W_3\end{bmatrix} \quad\text{ with }\quad W_i=\begin{bmatrix} W_{i,1} & W_{i,2}\end{bmatrix}:\BB{C}^2\mapsto\BB{C}^p.
\]
Then $W_i^*W_i=I_2$, hence $W_i$ is isometric, and $W_{i,1}^*W_{i,2}=0,$ for $i=1,2,3.$ The proof of the following result is based on the ideas behind the proof of the Dimension Theorem, but instead provides an explicit construction.

\begin{proposition}\label{P:K3rank2}
Let $X\in \PSD_G$, with $G\in\cK_3$, be extremal. Then $\rank(X)\leq 2$.
\end{proposition}

\begin{proof}[\bf Proof.]
Assume $X$ and $W$ are as in the paragraph preceding the proposition, as we may do without loss of generality. Assume $p=\rank (X)\geq 3$. We will construct a $p\times p$ matrix $R$ with $0\leq R\leq I$, $R\neq \la I$ for all $\la>0$, and $W_{i,1}^*RW_{i,2}=0,$ for $i=1,2,3$. Once such an $R$ is obtained, we can write $X$ as $X=X_1+X_2$ with $X_1=W^*RW$ and $X_2=W^*(I-R)W$ both in $\PSD_G$ and not scalar multiples of $X,$ contradicting the extremality of $X$. In other words, any extremal element of $\PSD_G$ must have rank at most 2.

If $p\geq 4$, simply take $v\perp W_{i,1},$ $i=1,2,3,$ and set $R=vv^*$. Hence, we need but consider the case $p=3$. Let $\BF=\BC$. We will first determine a Hermitian matrix $B$ such that $W_{i,1}^*BW_{i,2}=0$ for $i=1,2,3$ and $B\neq\la I,$ for all $\la\in\BR$. A  Hermitian matrix $B$ has the form
$$B=\begin{bmatrix}
r_1 & c_1 & c_2\\
\ol{c}_1 & r_2 & c_3\\
\ol{c}_2 & \ol{c}_3 & r_3
\end{bmatrix}$$
with $r_i\in\BR$ and $c_i\in\BC$, $i=1,2,3.$ Write out
$$W_{i,j}=\begin{bmatrix}
w_{i,j}^{(1)}\\
w_{i,j}^{(2)}\\
w_{i,j}^{(3)}
\end{bmatrix}.$$
Computing $W_{i,1}^*BW_{i,2}$ we obtain the following equation
\begin{align*}
W_{i,1}^*BW_{i,2} &=\ol{w}_{1,1}^{(1)}w_{2,2}^{(1)}r_1+\ol{w}_{1,1}^{(2)}w_{2,2}^{(2)}r_2+\ol{w}_{1,1}^{(3)}w_{2,2}^{(3)}r_3
+\ol{w}_{1,1}^{(1)}w_{2,2}^{(2)}c_1+\ol{w}_{1,1}^{(2)}w_{2,2}^{(1)}\ol{c}_1\\
&\quad+\ol{w}_{1,1}^{(1)}w_{2,2}^{(3)}c_2+\ol{w}_{1,1}^{(3)}w_{2,2}^{(1)}\ol{c}_2
+\ol{w}_{1,1}^{(2)}w_{2,2}^{(3)}c_3+\ol{w}_{1,1}^{(3)}w_{2,2}^{(2)}\ol{c}_3.
\end{align*}
Now, setting $W_{i,1}^*BW_{i,2}=0$ and splitting all elements in the equation into a real and imaginary part we obtain two homogeneous linear equations, with real coefficients and real variables. Doing this for $i=1,2,3$ we have six equations in nine real variables. Consequently, the solution space of this system of equations has a dimension of at least 3. Leaving out the scalar multiples of $I_3$, one still has a two dimensional space to chose an Hermitian matrix $B$ from with the desired properties. Now set $R=\frac{\rho I+B}{\|\rho I+B\|}$ with $\rho=\|B\|$. Then $R\geq 0$, $\|R\|=1$, so that $R\leq I$, and $R\neq \la I$ for all $\la\geq0$. Furthermore, since the solution set of the equations $W_{i,1}^* S W_{i,2}=0$, $i=1,2,3$, is linear and both $B$ and $I$ are solutions, so is $R$. We conclude $R$ has the the desired properties. This completes the proof for the case $\BF=\BC$. If $\BF=\BR$, in a similar way one obtains three linear equations with six variables, from which one can again construct an $R\geq 0$ with the required properties.
\end{proof}

As a side result of the above proof, we obtain the following corollary, cf., Lemma \ref{L:unitary} for the case $\cK_2$.

\begin{corollary}
Assume $X\in \PSD_G$, with $G\in\cK_3$, is of the form \eqref{Xdec2}, i.e., with $k=3$ and all entries of size $2\times 2$. If $X$ is extreme, then $C_{1,2}$, $C_{1,3}$ and $C_{2,3}$ are unitary. Conversely, if $C_{1,2}$, $C_{1,3}$ and $C_{2,3}$ are unitary, then $C_{1,2}^*C_{1,3}=C_{2,3}$, $X$ factors as
\[
X=\mat{c}{I\\ C_{1,2}\\ C_{2,3}}\mat{ccc}{I & C_{1,2} & C_{2,3}},
\]
and hence $\rank (X)= 2$.
\end{corollary}

\begin{proof}[\bf Proof.]
If $X\in \PSD_G$ is extremal, then $p=\rank(X)=2$. Hence the matrices $W_1$, $W_2$ and $W_3$ in the above proof are square isometries, hence unitary. Then $C_{1,2}=W_1^*W_2$,  $C_{1,3}=W_1^*W_3$ and $C_{2,3}=W_2^* W_3$ are also unitary.

Assume $C_{1,2}$, $C_{1,3}$ and $C_{2,3}$ are unitary. Again using the factorization $X=W^*W$, we have $C_{1,2}^*C_{1,3}=W_2^* W_1 W_1^* W_3=W_2^*W_3=C_{2,3}$. The factorization of $X$ now follows directly.
\end{proof}

\begin{proof}[\bf Proof of Theorem \ref{T:SO2complex} (iii) $\Rightarrow$ (i).]
Combining Propositions \ref{P:rank2} and \ref{P:K3rank2} we obtain that any extremal element of $\PSD_G$ for a $G\in \cK_2\cup \cK_3$ must have rank at most 2.
\end{proof}

\paragraph{\bf Acknowledgement}
The authors gratefully thank the anonymous referee for his/her useful comments and suggestions, in particular, for the proof of Proposition \ref{P:dim2subs} included in the paper, which is a significant improvement of the proof contained in the original manuscript.

This work is based on the research supported in part by the National Research Foundation of South Africa. Any opinion, finding and conclusion or recommendation expressed in this material is that of the authors and the NRF does not accept any liability in this regard.


\end{document}